\documentclass{article}

\usepackage[utf8]{inputenc}
\usepackage{authblk}
\usepackage{setspace}
\usepackage[margin=1.25in]{geometry}
\usepackage{graphicx}
\graphicspath{ {./figures/} }
\usepackage{subcaption}
\usepackage{amsmath}
\usepackage{lineno}
\usepackage{amsfonts}
\usepackage{amssymb, amsthm, amsmath, amsfonts}
\usepackage{wasysym}
\usepackage{mathrsfs}
\usepackage{hyperref}
\usepackage{graphicx}
\usepackage{lineno}
\usepackage[colorinlistoftodos]{todonotes}
\usepackage{listings}
\usepackage{cancel, enumerate}
\usepackage{rotating, environ}
\usepackage{caption}
\usepackage{subcaption}
\usepackage[inline]{enumitem}
\usepackage{dirtree}
\usepackage{xcolor}
\usepackage{tipa}
\usepackage{float}
\usepackage{changepage}

\newtheorem{defn}{Definition}
\newtheorem{prop}{Proposition}
\newtheorem{lemma}{Lemma}
\newtheorem{cor}{Corollary}

\newtheorem{rmk}{Remark}

\newtheorem{thm}{Theorem}

\renewcommand{\P}{\mathbb{P}}
\newcommand{\subseq}{\subseteq}

\newcommand{\Z}{\mathbb{Z}}

\newcommand{\vhi}{\varphi}

\newcommand{\kar}{\text{char}}

\renewcommand{\P}{\mathbb{P}}

\newcommand{\wilde}{\widetilde}

\newcommand{\Pic}{\text{Pic}}

\makeatletter
\newcommand{\tpitchfork}{%
  \vbox{
    \baselineskip\z@skip
    \lineskip-.52ex
    \lineskiplimit\maxdimen
    \m@th
    \ialign{##\crcr\hidewidth\smash{$-$}\hidewidth\crcr$\pitchfork$\crcr}
  }%
}
\makeatother

\newcommand{\Pcal}{\mathcal{P}}

\newcommand{\+}{\oplus}

\usepackage[style=numeric,citestyle=numeric-comp,backend=biber,maxbibnames=9]{biblatex}
\title{The classification of quasi-elliptic fibrations and unexpected plane cubics in characteristics 2 and 3}

\author{Jake Kettinger}

\affil{Colorado State University}

\date{27 October 2026}

\begin{document}

\maketitle

\begin{abstract}
In this paper, we categorize all isomorphism classes of quasi-elliptic surfaces over a field $k$ of characteristic 2 or 3. For every quasi-elliptic surface $X$, we classify all possible sequences of blow-downs from $X$ to the projective plane $\P^2_k$. We then use these categorizations to identify all unexpected plane cubic curves in characteristic 2 and present a proof of the lack of unexpected cubics in characteristic 3. Before the work in this paper-- based partly on the author's thesis-- the complete classification of unexpected plane cubic curves in characteristic 2 was unknown, as well as the question of the existence of unexpected plane cubic curves in characteristic 3.
\end{abstract}


\section{Introduction to quasi-elliptic fibrations}

This paper concerns the notion of a quasi-elliptic surface $X$. The main result of this paper is found in Theorem \ref{thm3} and Corollary \ref{no3} in section 4, wherein we prove there are no unexpected plane cubics in $\P^2_k$ for $\text{char}\,k=3$. In this section, we prove some general results about quasi-elliptic fibrations. In section 2, we describe in detail the Dynkin diagrams associated to each quasi-elliptic surface in characteristic 2. We use the Dynkin diagrams to investigate all possible ways of blowing down the surface $X$ to a plane, up to automorphism of $X$, providing explicit examples of each type of fibration. All the ways of blowing $X$ down to a plane then give us all possible labelings of the vertices of the Dynkin diagram of $X$, with labels as $(-2)$-curves of the Picard group, up to automorphism of $X$.  In section 3, we will collect all the data acquired about the blow-downs seen in section 2 to determine which blow-downs yield unexpected plane cubics in characteristic 2. 

In section 4, we describe the Dynkin diagrams associated to quasi-elliptic surfaces in characteristic 3, and perform the same analysis as in section 2 to classify all possible ways to blow up a plane to get a quasi-elliptic surface, up to automorphism of $X$. Finally, we investigate each possible manner of blowing up $\P^2$ to a quasi-elliptic surface and go on to show that none of these quasi-elliptic surfaces yield an unexpected plane cubic curve, and indeed that there are no unexpected plane cubics in characteristic 3. 

This is related to work done in characteristic 0 by Zanardini \cite{Zanardini}, who had classified constructions of Halphen pencils of index 2; (quasi-)elliptic surfaces are examples of index-1 Halphen pencils \cite{Halphen}.
\begin{defn}\normalfont
Let $k$ be an algebraically closed field, and let $X$ be a smooth projective surface. A \textbf{quasi-elliptic fibration} $f:X\to\P^1_k$ is a morphism whose general fiber is isomorphic to a singular plane cubic curve. 
\end{defn}
Note that if $X$ is a surface admitting such a morphism $f$, we will often simply refer to $X$ as a quasi-elliptic fibration. It is known that quasi-elliptic fibrations can occur only in characteristics 2 and 3 \cite{BM,L}. Also, any quasi-elliptic surface is also the consecutive blowing-up of the projective plane $\P^2$ at nine (possibly infinitely-near) points \cite{HL}. The nine points are in such a configuration that $\dim|-K_X|=1$, where $K_X\in\Pic X$ is the canonical divisor of $X$; thus $X$ is Jacobian \cite{HM}.

    We will show that every quasi-elliptic fibration $X$ is extremal; that is, every quasi-elliptic fibration has only finitely many sections (equivalently, the Mordell-Weil group of sections $\text{MW}(X)$ is finite).
    
    Each section of $X$ corresponds to a linear combination of the base points under the group law of each fibre, with a fixed base point chosen as the group identity. We will show such a set of points is well-defined.
    \begin{lemma}\label{ex}
    Let $A,B,C,D\in F\subset \P^2$, a general plane cubic curve. Let $\langle P_1,\dots,P_n\rangle_{Q}$ denote the subgroup of $F$ generated by the points $P_1,\dots,P_n\in F$ with $Q\in F$ chosen as the zero point. Let $P_1+_Q P_2$, $-_Q P$, and $n\cdot_Q P$ denote addition, negation, and $\Z$-module multiplication with $Q$ as the zero point, respectively. Then $A+_B C\in\langle A,B,C\rangle_D$, $A+_B C\in\langle A,B\rangle_C$, $-_A B\in\langle A\rangle_B$, and $-_A B\in\langle A,B\rangle_C$.
    \end{lemma}
    \begin{proof}
    All the claims in the lemma follow quickly from the bijection $\Pic^1 F\leftrightarrow F$. We will show that specifically that $P:=A+_B C=A+_D C-_D B$. Note that by the correspondence of the group $F$ with $\Pic^0 F$, we have $(A-B)+(C-B)\sim P-B$ by Example IV.1.3.7 of \cite{Hartshorne}. Then $A+C-B\sim P$. Moreover, we have $A+_D C-_D B=(A-D)+(C-D)-(B-D)+D=A+C-B\sim P$. Thus we have \[
    (A-B)+(C-B)+B\sim(A-D)+(C-D)-(B-D)+D\sim P
    \] and by the bijection $\Pic^1 F\leftrightarrow F$, we know that $P$ is the unique point on $F$ satisfying these linear equivalences.

    The other three claims can be refined to $A+_B C=A-_C B$, $-_B A=A+_B A$, and $-_B A=A+_C A-_C B$, which all follow similarly from the bijection $\Pic^1 F\leftrightarrow F$.
      \end{proof}
    Thus we immediately have the following corollaries.
    \begin{cor}
    For any nine points $P_1,\dots,P_9\in F$, we have $\langle P_1,\dots,P_9\rangle_{P_i}=\langle P_1\dots,P_9\rangle_{P_j}$ for all $1\leq i,j\leq 9$.
    \end{cor}
    
    \begin{cor}\label{extremal}
    Let $X$ be a blow-up of $\P^2$ that gives a quasi-elliptic fibration. Then $X$ is extremal.
    \end{cor}
    \begin{proof}
   By Lemma \ref{ex}, it is well-defined to say each section of $X$ corresponds to a linear combination of the base points under the group law of each fibre, with some fixed base point chosen as the zero.
    
    Let $F_\eta$ be the fiber of $\eta\in\P^1_k$, and define $G_\eta$ as $\langle P_1,\dots,P_9\rangle_{P_i}$ as a subgroup of $F_\eta$. Since each fibre $F_\eta$ is of additive type \cite{HL}, and $\kar\, k\in\{2,3\}$, the subgroup $G_\eta$ is a set of linear combinations of finitely many $\Z/2\Z$- or $\Z/3\Z$-vectors and is thus finite. Since the $(-1)$-curves are exactly the sections by \cite{HM}, they correspond to a curve given by one point on each $G_\eta$ as $\eta$ runs through $\P^1_k$. By the definition of the Mordell-Weil group of sections (given in e.g. the beginning of Section 2 of Lecture VII of \cite{Miranda89}), the sections of $X$ are in bijection with $G_\eta$, which is finite, and so there must be finitely many $(-1)$-curves. Therefore $X$ is extremal.
    \end{proof}
    
    The following lemma will be used to enumerate the number of $(-1)$-curves of every quasi-elliptic surface we will see in this paper.
    
    \begin{lemma}\label{MW}
    Let $X$ be a (quasi-)elliptic surface with canonical divisor $K\in\Pic(X)$. Define one of the $(-1)$-curves $s_0$ of $X$ as the zero-section of $X$ in the Mordell-Weil group $\text{MW}(X)$. Then define $U=\langle K,s_0\rangle$ and $V\leq\Pic(X)$ to be generated by the $(-2)$-curves that do not intersect $s_0$. Then $V\leq U^{\perp}$ and $$\text{MW}(X)\cong U^{\perp}/V.$$ 
    \end{lemma}
    \begin{proof}
    Line VII.2.5 of Theorem VII.2.1 in \cite{Miranda89}.
    \end{proof}
    In particular, since $X$ is extremal, we have that the number of $(-1)$-curves is equal to the size of the quotient group $U^\perp/V$.

    \begin{defn}
    Let $Z\subseq \P^2_k$ be a set of points. Let $I_{Z+dP}\subseq k[\P^2]$ be the ideal $\bigcap_{q\in Z}I(q)\cap I(P)^d$. We say $Z$ admits an \textbf{unexpected curve} of degree $d+1$ if $$h^0\left(\P^2_k,I_{Z+dP}(d+1)\right)>\max\left(0,I_Z(d+1)-\binom{d+1}{2}\right),$$ for a general point $P\in\P^2_k$.
    \end{defn}
    
    For a finite set $Z\subseq \P^n_k$ for $n\geq 3$, we consider whether $Z$ admits an unexpected cone of degree $d$, which is when $$h^0\left(\P^n_k,I_{Z+dP}(d)\right)>\max\left(0,I_Z(d)-\binom{d+n-1}{n}\right)$$ for a general point $P\in\P^n_k$. We will borrow notation from Chiantini and Migliore \cite{CM}, who denote $Z$ admitting an unexpected cone of degree $d$ as $Z$ having the $C(d)$ property, in $\P^n_k$ (or $C_Z(d)$ if it is not clear which $Z$ is intended).
    
    In 2017, Solomon Akesseh \cite{A} proved that a finite set $Z\subseq\P^2$ of distinct points has an unexpected plane cubic curve only when $Z$ consists of the seven points of the Fano plane in characteristic 2. Farnik, Galuppi, Sodomaco, and Trok \cite{FGST} recovered the result that there is no unexpected cubic in characteristic 0, and in addition showed in characteristic 0 that unexpected quartics arise for a unique (up to choice of coordinates) configuration of 9 points in $\P^2$. Configurations that admit unexpected varieties in higher dimensions are connected to an interesting geometric property known as the geproci property \cite{sevenauthors}, which also has a home in the positive characteristic setting \cite{Kettinger}.

    To begin discussing unexpected cubic curves, it is necessary to detail their connection to quasi-elliptic fibrations with the following lemma. Note that throughout this paper, we refer to a set $Z$ of ``not necessarily distinct" points; this refers to the fact that some points of $Z$ may be infinitely-near other points of $Z$. Given an algebraic set $X$, a point $Q$ is \textbf{infinitely-near} the point $P$ if $Q$ maps to $P$ under the standard blowup map $\pi_P:\text{Bl}_P(X)\to X$, as in Chapter 4.2 of \cite{GH}. On the other hand, if $Q\in X\setminus {P}$, then $Q$ and $P$ are \textbf{distinct}. We will often denote an ordinary point of $\P^2_k$ as $P^{(1)}$ and denote a point infinitely-near $P^{(i)}$ as $P^{(i+1)}$, and the set of points $P^{(1)},P^{(2)},\dots,P^{(n)}$ as $P\times n$.
    
     
    \begin{lemma}\label{l2}
Let $Z$ be a set of seven (not necessarily distinct) points in $\P^2_k$ admitting an unexpected cubic. Then the points impose independent conditions on cubics, every cubic curve through $Z$ is singular, the general cubic through $Z$ is reduced and irreducible, and for a general point $P$ there is a unique cubic $C_P$ singular at $P$; it is reduced and irreducible.
\end{lemma}
\begin{proof}
The proof of Lemma 3.2.9 in Akesseh \cite{A} carries over to the case that $Z$ contains infinitely near points.
\end{proof}

\begin{prop}\label{uq}
Let $Z=\{P_1,\dots,P_7\}$ be a set of seven (not necessarily distinct) points in $\P^2_k$ admitting an unexpected cubic. Let $P\in\P^2$ be general. Then the set $\Pcal_P$ of cubic curves containing $Z\cup\{P\}$ is a quasi-elliptic fibration. Furthermore, this fibration can be viewed as a surface that is isomorphic to the blowup of $\P^2_k$ at nine points: the seven points of $Z$, the point $P$ and a point infinitely-near $P$.
\end{prop}
\begin{proof}
Since $Z$ imposes independent conditions on cubics by Lemma \ref{l2} and since $P$ is general, $\Pcal_P$ is a linear pencil of cubics. By Lemma \ref{l2}, the cubics are singular but fixed component free, so blowing up the base points gives a quasi-elliptic fibration. Since the general member of the pencil is a cubic, there are 9 base points. By Lemma \ref{l2}, the cubics containing $Z\cup\{P\}$ contain $C_P$ and the general one, $C$, is reduced and irreducible. Since $C.C_P=P_1+\cdots+P_7+2P$, the 9 base points are $P_1,\dots,P_7,P,P_9$, where $P_9$ is infinitely near $P$.\end{proof}



\section{Quasi-Elliptic Fibrations in Characteristic 2}

Given a quasi-elliptic fibration $X$, one may construct a multigraph $D$ the following way: the vertex set is the set of curves on $X$ with a self-intersection of $-2$, and the number of edges connecting two vertices is equal to the multiplicity of the intersection of their respective curves. The multigraph $D$ constructed this way is called the \textbf{Dynkin diagram} of $X$.

The possible Dynkin diagrams for quasi-elliptic fibrations in characteristic 2 are $\wilde{A}_1^{\+8}$, $\wilde{A}_1^{\+4}\+\wilde{D}_4$, $\wilde{A}_1^{\+2}\+\wilde{D}_6$, $\wilde{D}_4^{\+2}$, $\wilde{A}_1\+\wilde{E}_7$, $\wilde{D}_8$, and $\wilde{E}_8$ \cite{CD}, Proposition 5.6.3, or \cite{CDL} Table 4.9. We will show that there are 2 ways to blow down $\wilde{A}_1^{\+8}$ to $\P^2$, 4 ways to blow down $\wilde{A}_1^{\+4}\+\wilde{D}_4$, 4 ways to blow down $\wilde{A}_1^{\+2}\+\wilde{D}_6$, 2 ways to blow down $\wilde{D}_4^{\+2}$, 2 ways to blow down $\wilde{A}_1\+\wilde{E}_7$, 2 ways to blow down $\wilde{D}_8$, and 1 way to blow down $\wilde{E}_8$.
\begin{center}
    \textbf{Case 1: $\wilde{A}_1^{\+8}$}
\end{center}

Let us start with $\wilde{A}_1^{\+8}$. We can get this Dynkin diagram by blowing up $P_1^{(1)}$, $P_1^{(2)}$, $P_2^{(1)}$, $P_2^{(2)}$, $P_3^{(1)}$, $P_3^{(2)}$, $P_4^{(1)}$, $P_4^{(2)}$, $P_5$ where $P_n^{(2)}$ is infinitely-near $P_n^{(1)}$ in the direction of $P_5$. Let us call this surface $X$. Note $\Pic X=\langle \ell,e_1,\dots,e_9\rangle$. We will draw the Dynkin diagram below:

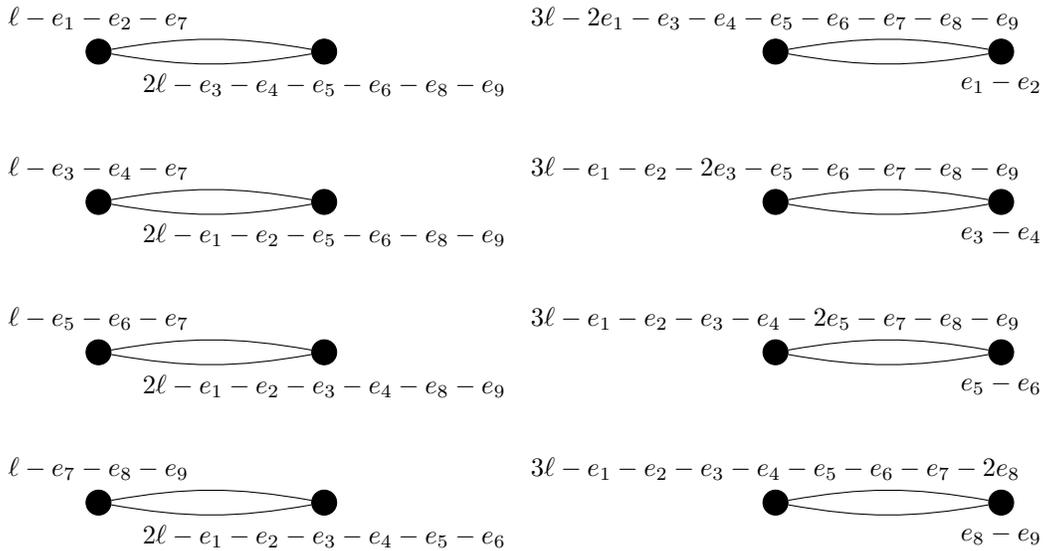
\begin{tikzpicture}
    \node[shape=circle,draw=black,fill=black,label=$\ell-e_1-e_2-e_7$] (F) at (0,0) {} ;
    \node[shape=circle,draw=black,fill=black,label=below:$2\ell-e_3-e_4-e_5-e_6-e_8-e_9$] (G) at (3,0) {} ;
    \node[shape=circle,draw=black,fill=black,label=$\ell-e_3-e_4-e_7$] (H) at (0,-2) {} ;
    \node[shape=circle,draw=black,fill=black,label=below:$2\ell-e_1-e_2-e_5-e_6-e_8-e_9$] (I) at (3,-2) {} ;
    \node[shape=circle,draw=black,fill=black,label=$\ell-e_5-e_6-e_7$] (J) at (0,-4) {} ;
    \node[shape=circle,draw=black,fill=black,label=below:$2\ell-e_1-e_2-e_3-e_4-e_8-e_9$] (K) at (3,-4) {} ;
    \node[shape=circle,draw=black,fill=black,label=$\ell-e_7-e_8-e_9$] (L) at (0,-6) {} ;
    \node[shape=circle,draw=black,fill=black,label=below:$2\ell-e_1-e_2-e_3-e_4-e_5-e_6$] (M) at (3,-6) {} ;
    
    \node[shape=circle,draw=black,fill=black,label=$3\ell-2e_1-e_3-e_4-e_5-e_6-e_7-e_8-e_9$] (A) at (9,0) {} ;
    \node[shape=circle,draw=black,fill=black,label=below:$e_1-e_2$] (B) at (12,0) {} ;
    \node[shape=circle,draw=black,fill=black,label=$3\ell-e_1-e_2-2e_3-e_5-e_6-e_7-e_8-e_9$] (C) at (9,-2) {} ;
    \node[shape=circle,draw=black,fill=black,label=below:$e_3-e_4$] (D) at (12,-2) {} ;
    \node[shape=circle,draw=black,fill=black,label=$3\ell-e_1-e_2-e_3-e_4-2e_5-e_7-e_8-e_9$] (E) at (9,-4) {} ;
    \node[shape=circle,draw=black,fill=black,label=below:$e_5-e_6$] (N) at (12,-4) {} ;
    \node[shape=circle,draw=black,fill=black,label=$3\ell-e_1-e_2-e_3-e_4-e_5-e_6-e_7-2e_8$] (O) at (9,-6) {} ;
    \node[shape=circle,draw=black,fill=black,label=below:$e_8-e_9$] (P) at (12,-6) {} ;

    \path [-,draw=black, bend left=10](F) edge node[left] {} (G);
    \path [-,draw=black, bend left=10](H) edge node[left] {} (I);
    \path [-,draw=black, bend left=10](J) edge node[left] {} (K);
    \path [-,draw=black, bend left=10](L) edge node[left] {} (M);
    \path [-,draw=black, bend left=10](A) edge node[left] {} (B);
    \path [-,draw=black, bend left=10](C) edge node[left] {} (D);
    \path [-,draw=black, bend left=10](E) edge node[left] {} (N);
    \path [-,draw=black, bend left=10](O) edge node[left] {} (P);
    
    \path [-,draw=black, bend right=10](F) edge node[left] {} (G);
    \path [-,draw=black, bend right=10](H) edge node[left] {} (I);
    \path [-,draw=black, bend right=10](J) edge node[left] {} (K);
    \path [-,draw=black, bend right=10](L) edge node[left] {} (M);
    \path [-,draw=black, bend right=10](A) edge node[left] {} (B);
    \path [-,draw=black, bend right=10](C) edge node[left] {} (D);
    \path [-,draw=black, bend right=10](E) edge node[left] {} (N);
    \path [-,draw=black, bend right=10](O) edge node[left] {} (P);
\end{tikzpicture}


Using Lemma \ref{MW}, we know that there are 16 $(-1)$-curves. The 16 $(-1)$-curves are $e_2$, $e_4$, $e_6$, $e_7$, $e_9$, $\ell-e_1-e_3$, $\ell-e_1-e_5$, $\ell-e_1-e_8$, $\ell-e_3-e_5$, $\ell-e_5-e_8$, $2\ell-e_1-e_2-e_3-e_5-e_8$, $2\ell-e_1-e_3-e_4-e_5-e_8$, $2\ell-e_1-e_3-e_5-e_6-e_8$, $2\ell-e_1-e_3-e_5-e_7-e_8$, and $2\ell-e_1-e_3-e_5-e_8-e_9$. 

Taking the $(-2)$-curves and the $(-1)$-curves together, we can form a multigraph on 32 vertices that yields the following adjacency matrix.
\setcounter{MaxMatrixCols}{20}
$$\mathbf{\Gamma}=\begin{pmatrix}\mathbf{B}&\mathbf{M}\\ \mathbf{M}^\mathsf{T}&\mathbf{R}\end{pmatrix},$$ where
\begin{center}
    $$\mathbf{B}=-2\mathbf{I}_{16}+2\mathbf{J}_{16},$$
\end{center}
represents the intersection products between the $(-2)$-curves ($\mathbf{I}_{16}$ is the $16\times 16$ identity matrix and $\mathbf{J}_{16}$ is the $16\times 16$ exchange matrix),
\begin{center}
    $$\mathbf{R}=-\mathbf{I}_{16}+\mathbf{J}_{16}$$
\end{center}
represents the intersection products between the $(-1)$-curves, and $$\mathbf{M}=\left(\begin{matrix}
0 & 1 & 0 & 1 & 1 & 1 & 1 & 1 & 0 & 0 & 0 & 0 & 0 & 1 & 0 & 1 \\
1 & 0 & 0 & 1 & 1 & 1 & 0 & 0 & 1 & 1 & 0 & 0 & 0 & 1 & 1 & 0 \\
0 & 0 & 1 & 1 & 0 & 1 & 0 & 1 & 0 & 1 & 0 & 1 & 0 & 0 & 1 & 1 \\
1 & 1 & 0 & 1 & 0 & 0 & 0 & 1 & 0 & 1 & 1 & 1 & 0 & 1 & 0 & 0 \\
0 & 1 & 1 & 1 & 1 & 0 & 0 & 0 & 1 & 1 & 1 & 0 & 0 & 0 & 0 & 1 \\
1 & 1 & 1 & 1 & 0 & 1 & 1 & 0 & 1 & 0 & 0 & 1 & 0 & 0 & 0 & 0 \\
0 & 0 & 0 & 1 & 0 & 0 & 1 & 0 & 1 & 0 & 1 & 1 & 0 & 1 & 1 & 1 \\
1 & 0 & 1 & 1 & 1 & 0 & 1 & 1 & 0 & 0 & 1 & 0 & 0 & 0 & 1 & 0 \\
0 & 1 & 0 & 0 & 0 & 1 & 0 & 0 & 1 & 1 & 0 & 1 & 1 & 1 & 0 & 1 \\
1 & 1 & 1 & 0 & 1 & 1 & 0 & 1 & 0 & 1 & 0 & 0 & 1 & 0 & 0 & 0 \\
0 & 0 & 0 & 0 & 1 & 0 & 0 & 1 & 0 & 1 & 1 & 0 & 1 & 1 & 1 & 1 \\
1 & 0 & 0 & 0 & 0 & 1 & 1 & 1 & 0 & 0 & 0 & 1 & 1 & 1 & 1 & 0 \\
0 & 0 & 1 & 0 & 1 & 1 & 1 & 0 & 1 & 0 & 0 & 0 & 1 & 0 & 1 & 1 \\
1 & 1 & 0 & 0 & 1 & 0 & 1 & 0 & 1 & 0 & 1 & 0 & 1 & 1 & 0 & 0 \\
0 & 1 & 1 & 0 & 0 & 0 & 1 & 1 & 0 & 0 & 1 & 1 & 1 & 0 & 0 & 1 \\
1 & 0 & 1 & 0 & 0 & 0 & 0 & 0 & 1 & 1 & 1 & 1 & 1 & 0 & 1 & 0
\end{matrix}\right)$$ 
represents the intersection products between the $(-1)$- and $(-2)$-curves. 



\begin{prop}\label{blowdowns}
There are-- up to isomorphism of the adjacency diagram of $X$-- two ways to blow down $X$ to the projective plane $\P^2$.
\end{prop}
\begin{proof}
By Castelnuovo's contraction theorem, one can blow down a curve $C$ on a smooth projective surface $X$ to get a smooth projective surface if and only if the self-intersection number of $C$ is $-1$ \cite{Hartshorne} V.5.7.2. Let us denote by $B$ the set of $(-2)$-curves and by $R$ the set of $(-1)$-curves.

First choose one $(-1)$-curve $R_1$ to blow down. Note that since we are blowing down to $\P^2$, we will need to blow down exactly nine curves. Once we blow down a $(-1)$-curve, the single adjacent $(-1)$-curve as given by $\mathbf{\Gamma}$ cannot be blown down since its self-intersection number would increase. There are eight pairs of adjacent $(-1)$ curves. Therefore, if we blow down eight disjoint $(-1)$-curves, we would still need to blow down one of the $(-2)$-curves as the ninth. Thus there must be at least one $(-2)$-vertex present in the blow-down. Choose $B_1$ adjacent to $R_1$ to blow down.

Now let us take an inventory of which curves cannot be blown down. The single $(-2)$-curve adjacent to $B_1$ cannot be blown down, because it intersects $B_1$ with multiplicity 2, so its self-intersection number has increased to 0. In addition, the other 7 $(-1)$-curves adjacent to $B_1$ cannot be blown down, because their self-intersection have also increased to 0. Finally, the other 7 $(-2)$-curves adjacent to $R_1$ cannot be blown down, because their self intersection increases to $(-1)$ after $R_1$ is blown down, but since $B_1$ also intersects $R_1$, their self intersection would then increase to 0 after $B_1$ is blown down.

What is left of $\mathbf{\Gamma}$ after deleting all the vertices whose self-intersection number has risen to a nonnegative value after blowing down $R_1$ and $B_1$ is a three-regular bipartite graph on 14 vertices containing 7 disjoint $(-1)$-vertices and 7 disjoint $(-2)$-vertices. 

In this subgraph, the two possibilities for blow-downs are to blow down the 7 disjoint $(-1)$-vertices, or to blow down one $(-1)$-vertex and three $(-1)$-$(-2)$ pairs.

It is impossible to blow down three disjoint $(-1)$ vertices and two $(-1)$-$(-2)$ pairs because there do not exist three $(-1)$-vertices that no two $(-2)$-vertices are adjacent to.

It is impossible to blow down five disjoint $(-1)$-vertices and one $(-1)$-$(-2)$ pair because every $(-2)$-vertex is adjacent to three $(-1)$-vertices: if one has already been blown down, then there are two of the remaining six connected to it, leaving four out of six not connected to it. This is less than the five required, so this blow down is impossible. Thus there are only two types of blow-downs.
\end{proof}
\begin{rmk}\label{algo}
Note that we can run this procedure algebraically by performing the following computations on the adjacency matrix $\mathbf{\Gamma}$. We use the fact that blowing down the curve $R_1$ increases the self-intersection number of a curve $C$ by $(C.R_1)^2$ and increases the intersection product of two general divisors $C$ and $D$ by $(C.R_1)\cdot(D.R_1)$, \cite{Hartshorne} Proposition V.3.2(d) and Proposition V.3.6. We can thus form a new graph $\mathbf{\Gamma}_{R_1}$ induced by blowing down $R_1$ by using the column $\mathbf{r}_1$ of $\mathbf{\Gamma}$ corresponding to the vertex representing $R_1$ and computing $\mathbf{r}_1\mathbf{r}_1^{\mathsf{T}}+\mathbf{\Gamma}$. Then the row and column corresponding to $R_1$ are now completely zero, and the $(i,j)$-entry of $\mathbf{\Gamma}_{R_1}$ is the intersection product of the curves corresponding to $\mathbf{\Gamma}_i$ and $\mathbf{\Gamma}_j$ after having blown down $R_1$. We may repeat this procedure until we blow down to $\P^2$, when $\mathbf{\Gamma}$ will have been transformed into a matrix with no negative entries. This fact is the main principle behind the Macaulay2 code in the Appendix that will be able to compute the blow-downs of the sufficiently small Dynkin diagrams in this paper.
\end{rmk}

The two blow-downs are given by the following diagrams. The $(-1)$-curves are illustrated by the hollow vertices, and the $(-2)$-curves are illustrated by the solid black vertices.
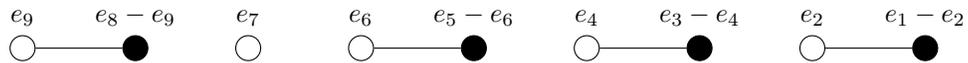
\begin{center}
\begin{tikzpicture}
    \node[shape=circle,draw=black,fill=none, label=$e_9$] (A) at (0,0) {} ;
    \node[shape=circle,draw=black,fill=black, label=$e_8-e_9$] (B) at (1.5,0) {} ;
    \node[shape=circle,draw=black,fill=none,label=$e_7$] (C) at (3,0) {} ;
    \node[shape=circle,draw=black,fill=none, label=$e_6$] (D) at (4.5,0) {} ;
    \node[shape=circle,draw=black,fill=black, label=$e_5-e_6$] (E) at (6,0) {} ;
    \node[shape=circle,draw=black,fill=none, label=$e_4$] (F) at (7.5,0) {} ;
    \node[shape=circle,draw=black,fill=black, label=$e_3-e_4$] (G) at (9,0) {} ;
    \node[shape=circle,draw=black,fill=none, label=$e_2$] (H) at (10.5,0) {} ;
    \node[shape=circle,draw=black,fill=black, label=$e_1-e_2$] (I) at (12,0) {} ;

    \path [-,draw=black](A) edge node[left] {} (B);
    \path [-,draw=black](D) edge node[left] {} (E);
    \path [-,draw=black](F) edge node[left] {} (G);
    \path [-,draw=black](H) edge node[left] {} (I);
\end{tikzpicture}
\end{center}

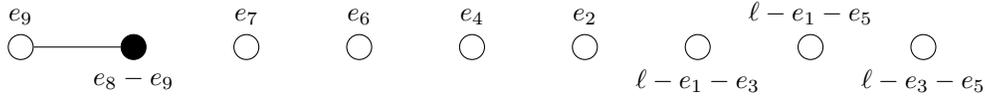
\begin{center}
\begin{tikzpicture}
    \node[shape=circle,draw=black,fill=none, label=$e_9$] (A) at (0,0) {} ;
    \node[shape=circle,draw=black,fill=black, label=below:$e_8-e_9$] (B) at (1.5,0) {} ;
    \node[shape=circle,draw=black,fill=none, label=$e_7$] (C) at (3,0) {} ;
    \node[shape=circle,draw=black,fill=none, label=$e_6$] (D) at (4.5,0) {} ;
    \node[shape=circle,draw=black,fill=none, label=$e_4$] (E) at (6,0) {} ;
    \node[shape=circle,draw=black,fill=none, label=$e_2$] (F) at (7.5,0) {} ;
    \node[shape=circle,draw=black,fill=none, label=below:$\ell-e_1-e_3$] (G) at (9,0) {} ;
    \node[shape=circle,draw=black,fill=none, label=$\ell-e_1-e_5$] (H) at (10.5,0) {} ;
    \node[shape=circle,draw=black,fill=none, label=below:$\ell-e_3-e_5$] (I) at (12,0) {} ;

    \path [-,draw=black](A) edge node[left] {} (B);
\end{tikzpicture}
\end{center}
The first blow-down diagram corresponds to the exact points on $\P^2$ that were blown up to obtain the surface $X$ in the first place.

The second diagram induces a map $A:\Pic X\longrightarrow \Pic X$ given by the matrix $$A=\left(\begin{matrix}
2 & 1 & 0 & 1 & 0 & 1 & 0 & 0 & 0 & 0 \\
-1 & 0 & 0 & -1 & 0 & -1 & 0 & 0 & 0 & 0 \\
0 & 0 & 1 & 0 & 0 & 0 & 0 & 0 & 0 & 0 \\
-1 & -1 & 0 & 0 & 0 & -1 & 0 & 0 & 0 & 0 \\
0 & 0 & 0 & 0 & 1 & 0 & 0 & 0 & 0 & 0 \\
-1 & -1 & 0 & -1 & 0 & 0 & 0 & 0 & 0 & 0 \\
0 & 0 & 0 & 0 & 0 & 0 & 1 & 0 & 0 & 0 \\
0 & 0 & 0 & 0 & 0 & 0 & 0 & 1 & 0 & 0 \\
0 & 0 & 0 & 0 & 0 & 0 & 0 & 0 & 1 & 0 \\
0 & 0 & 0 & 0 & 0 & 0 & 0 & 0 & 0 & 1
\end{matrix}\right).$$

The map $A$ sends each $(-1)$-curve shown in the second blow-down diagram to a $(-1)$-curve of the form $e_i$ and each $(-2)$-curve shown to a $(-2)$-curve of the form $e_i-e_j$. This allows us to express all curves in $\Pic(X)$ as $\Z$-linear combinations of the curves of the second blow-down diagram, thus allowing us to learn what configuration of points in $\P^2_k$ we are required to blow up to attain $X$. In section 3, we will determine which configurations yield unexpected cubics.

Applying $A$ to the Dynkin diagram gives us a new Dynkin diagram shown below. To save space, we will denote $3\ell-e_1-e_2-e_3-e_4-e_5-e_6-e_7-2e_8$ as $-K-e_7+e_8$, where $K=-3\ell+e_1+\cdots+e_9$ is the canonical divisor and rename $2\ell-e_{i_1}-e_{i_2}-e_{i_3}-e_{i_4}-e_{i_5}-e_{i_6}$ as $2\ell-i_1i_2i_3i_4i_5i_6$, and $\ell-e_{i_1}-e_{i_2}-e_{i_3}$ as $\ell-i_1i_2i_3$.

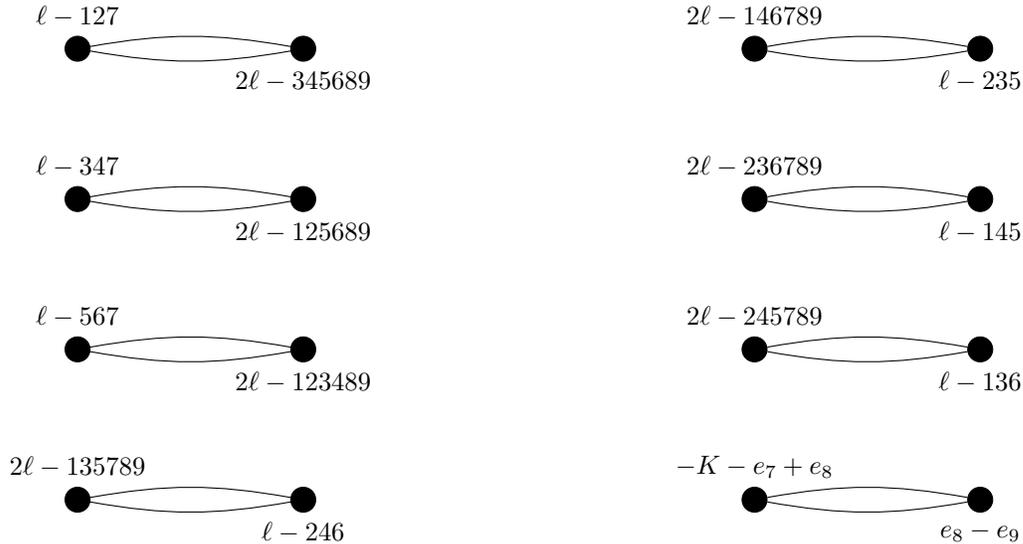
\begin{adjustwidth*}{}{-0em}
\begin{tikzpicture}
    \node[shape=circle,draw=black,fill=black,label=$\ell-127$] (F) at (0,0) {} ;
    \node[shape=circle,draw=black,fill=black,label=below:$2\ell-345689$] (G) at (3,0) {} ;
    \node[shape=circle,draw=black,fill=black,label=$\ell-347$] (H) at (0,-2) {} ;
    \node[shape=circle,draw=black,fill=black,label=below:$2\ell-125689$] (I) at (3,-2) {} ;
    \node[shape=circle,draw=black,fill=black,label=$\ell-567$] (J) at (0,-4) {} ;
    \node[shape=circle,draw=black,fill=black,label=below:$2\ell-123489$] (K) at (3,-4) {} ;
    \node[shape=circle,draw=black,fill=black,label=$2\ell-135789$] (L) at (0,-6) {} ;
    \node[shape=circle,draw=black,fill=black,label=below:$\ell-246$] (M) at (3,-6) {} ;
    
    \node[shape=circle,draw=black,fill=black,label=$2\ell-146789$] (A) at (9,0) {} ;
    \node[shape=circle,draw=black,fill=black,label=below:$\ell-235$] (B) at (12,0) {} ;
    \node[shape=circle,draw=black,fill=black,label=$2\ell-236789$] (C) at (9,-2) {} ;
    \node[shape=circle,draw=black,fill=black,label=below:$\ell-145$] (D) at (12,-2) {} ;
    \node[shape=circle,draw=black,fill=black,label=$2\ell-245789$] (E) at (9,-4) {} ;
    \node[shape=circle,draw=black,fill=black,label=below:$\ell-136$] (N) at (12,-4) {} ;
    \node[shape=circle,draw=black,fill=black,label=$-K-e_7+e_8$] (O) at (9,-6) {} ;
    \node[shape=circle,draw=black,fill=black,label=below:$e_8-e_9$] (P) at (12,-6) {} ;

    \path [-,draw=black, bend left=10](F) edge node[left] {} (G);
    \path [-,draw=black, bend left=10](H) edge node[left] {} (I);
    \path [-,draw=black, bend left=10](J) edge node[left] {} (K);
    \path [-,draw=black, bend left=10](L) edge node[left] {} (M);
    \path [-,draw=black, bend left=10](A) edge node[left] {} (B);
    \path [-,draw=black, bend left=10](C) edge node[left] {} (D);
    \path [-,draw=black, bend left=10](E) edge node[left] {} (N);
    \path [-,draw=black, bend left=10](O) edge node[left] {} (P);
    
    \path [-,draw=black, bend right=10](F) edge node[left] {} (G);
    \path [-,draw=black, bend right=10](H) edge node[left] {} (I);
    \path [-,draw=black, bend right=10](J) edge node[left] {} (K);
    \path [-,draw=black, bend right=10](L) edge node[left] {} (M);
    \path [-,draw=black, bend right=10](A) edge node[left] {} (B);
    \path [-,draw=black, bend right=10](C) edge node[left] {} (D);
    \path [-,draw=black, bend right=10](E) edge node[left] {} (N);
    \path [-,draw=black, bend right=10](O) edge node[left] {} (P);
\end{tikzpicture}
\end{adjustwidth*}
\begin{itemize}
\item The original Dynkin diagram corresponds to the blowup of $\P^2$ at five ordinary points $P_1^{(1)}$ through $P_5^{(1)}$ and four infinitely-near points $P_1^{(2)},\dots,P_4^{(2)}$ where each infinitely-near point $P_i^{(2)}$ corresponds to the line containing the ordinary points $P_i^{(1)}$ and $P_5$. An example of a pencil of cubics with base locus at such an arrangement is that spanned by $a^2(xy+xz+yz)+(ab+ac+bc)x^2$ and $b^2(xy+xz+yz)+(ab+ac+bc)y^2$, where $(a,b,c)\in\P^2_{\overline{k}}$ may be a generic point.

\item The second diagram corresponds to seven points blown up in a Fano plane configuration, and a general eighth point blown up twice. An example of a pencil of cubics whose base locus is such an arrangement is that spanned by $(ac+c^2)(x^2y+xy^2)+(ab+b^2)(x^2z+xz^2)$ and $(bc+c^2)(x^2y+xy^2)+(ab+b^2)(y^2z+yz^2)$, where $(a,b,c)\in\P^2_{\overline{k}}$ may be a generic point.

\end{itemize}

\begin{center}
    \textbf{Case 2:} $\wilde{A}_1^{\+4}\+\wilde{D}_4$
\end{center}

Let's now look at the Dynkin diagram $\wilde{A}_1^{\+4}\+\wilde{D}_4$, as depicted below. 

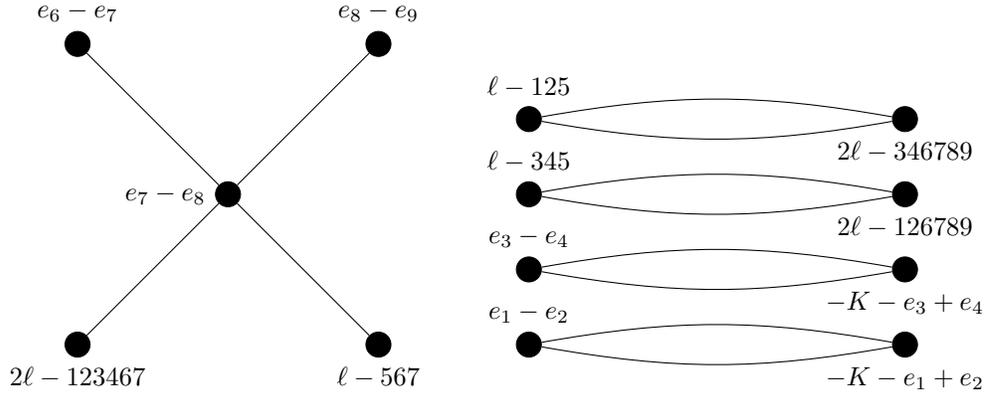
\begin{adjustwidth*}{}{-0em}
\begin{tikzpicture}
    \node[shape=circle,draw=black,fill=black,label=$e_6-e_7$] (A) at (0,0) {};
    \node[shape=circle,draw=black,fill=black,label=$e_8-e_9$] (B) at (4,0) {};
    \node[shape=circle,draw=black,fill=black,label=left:$e_7-e_8$] (C) at (2,-2) {};
    \node[shape=circle,draw=black,fill=black,label=below:$2\ell-123467$] (D) at (0,-4) {};
    \node[shape=circle,draw=black,fill=black,label=below:$\ell-567$] (E) at (4,-4) {};
    \node[shape=circle,draw=black,fill=black,label=$\ell-125$] (F) at (6,-1) {} ;
    \node[shape=circle,draw=black,fill=black,label=below:$2\ell-346789$] (G) at (11,-1) {} ;
    \node[shape=circle,draw=black,fill=black,label=$\ell-345$] (H) at (6,-2) {} ;
    \node[shape=circle,draw=black,fill=black,label=below:$2\ell-126789$] (I) at (11,-2) {} ;
    \node[shape=circle,draw=black,fill=black,label=$e_3-e_4$] (J) at (6,-3) {} ;
    \node[shape=circle,draw=black,fill=black,label=below:$-K-e_3+e_4$] (K) at (11,-3) {} ;
    \node[shape=circle,draw=black,fill=black,label=$e_1-e_2$] (L) at (6,-4) {} ;
    \node[shape=circle,draw=black,fill=black,label=below:$-K-e_1+e_2$] (M) at (11,-4) {} ;

    \path [-](B) edge node[left] {} (C);
    \path [-](A) edge node[left] {} (C);
    \path [-](D) edge node[left] {} (C);
    \path [-](E) edge node[left] {} (C);
    \path [-,draw=black, bend left=10](F) edge node[left] {} (G);
    \path [-,draw=black, bend right=10](F) edge node[left] {} (G);
    \path [-,draw=black, bend left=10](H) edge node[left] {} (I);
    \path [-,draw=black, bend right=10](H) edge node[left] {} (I);
    \path [-,draw=black, bend left=10](J) edge node[left] {} (K);
    \path [-,draw=black, bend right=10](J) edge node[left] {} (K);
    \path [-,draw=black, bend left=10](L) edge node[left] {} (M);
    \path [-,draw=black, bend right=10](L) edge node[left] {} (M);
    
\end{tikzpicture}
\end{adjustwidth*}

Using Lemma \ref{MW}, we know there are eight $(-1)$-curves; they are $e_2$, $e_4$, $e_5$, $e_9$, $\ell-e_1-e_3$, $\ell-e_1-e_6$, $\ell-e_3-e_6$, and $2\ell-e_1-e_3-e_6-e_7-e_8$. 

We will show that there are-- up to isomorphism of the graph-- four ways to blow down the surface to $\P^2$. We have eight $(-1)$-curves which are all disjoint from each other, and 13 $(-2)$-curves. To avoid ambiguity with $(-2)$-curves whose self-intersection may be raised to $-1$ during the blow-down process, we shall hereafter refer to these starting $(-1)$-curves as ``sections." We cannot blow down all eight sections, because then all the $(-2)$-curves except $e_7-e_8$ will have their self-intersection increase to 0; by contrast, the self-intersection of $e_7-e_8$ will remain $(-2)$, so nothing more can be blown down. Therefore, let us blow down seven of the sections. Other than the eighth section, the only curve of self-intersection $-1$ left is the one curve on the $\wilde{D}_4$ that intersects the eighth section that has not been blown down. Upon blowing down this curve, the self-intersection of $e_7-e_8$ increases to $-1$, and so we may now blow down this as our final $(-1)$-curve.

Next, we may consider blowing down six sections. We will again consider how they intersect the $(-2)$-curves on the $\wilde{D}_4$, where there are two possible arrangements: either all but one exterior vertices on the $\wilde{D}_4$ meet the six blown down sections (with the remaining vertex meeting none), or where two $(-2)$-curves meet two sections and the other two meet one. In the former case, we may not blow down any vertices of the $\wilde{D}_4$, as they all have their self-intersections either increased to 0 or left at $-2$. Similarly, we may not blow down any vertices of any of the $\wilde{A}_1$'s, which all intersect exactly three of the six chosen sections, and thus have had their self-intersection increase to 1. Therefore we may not proceed with this case. In the latter case, all vertices on an $\wilde{A}_1$ have self-intersection either 0 or 1, and so we may only proceed by blowing down one of the two $(-1)$-curves on the $\wilde{D}_4$. If we blow down both, then $e_7-e_8$ will have its self-intersection increase to 0, and the other two original sections will have their self-intersection increase to 0, so we may not proceed. Therefore, the only way to proceed is to blow down just one of the $(-1)$-curves on the $\wilde{D}_4$, followed by the $e_7-e_8$, and finally one of the sections, giving us the same blow-down sequence as before.

Next, let us consider blowing down five of the sections. The five sections can chosen to meet the exterior vertices of the $\wilde{D}_4$ with intersections 2, 2, 1, and 0, or with 2, 1, 1, and 1. In the former case, each vertex of an $\wilde{A}_1$ now has self-intersection either 0 or 1, and so may not be blown down. The only vertex on the $\wilde{D}_4$ with a self-intersection of $-1$ is the one that meets only one of the five blown-down sections. Upon blowing down that vertex, we may now proceed to blow down $e_7-e_8$ and then the single vertex of the $\wilde{D}_4$ that does not intersect any of the five sections. Then we have blown down eight total curves, but we may not proceed, as there are no more $(-1)$-curves. In the 2, 1, 1, 1-case, there is precisely one vertex on an $\wilde{A}_1$ whose self-intersection has increased to $-1$, and three vertices on the $\wilde{D}_4$ whose self-intersection is $-1$. We must blow down exactly these four. Blowing down the $e_7-e_8$ will cause the other exterior vertices to have their self-intersection increase above $-1$, and we will not be able to blow down nine curves.

Next, let us consider blowing down four sections. The four curves can be chosen to meet the exterior vertices of the $\wilde{D}_4$ with intersections 1, 1, 1, and 1, or 2, 1, 1, and 0, or 2, 2, 0, and 0. In the first case, we may not blow down the four exterior $\wilde{D}_4$ vertices, because the self-intersection of $e_7-e_8$ will then increase to 2, and all vertices on the $\wilde{A}_1$'s will increase above $-1$ as well. For the same reason, we cannot blow down three or two of the exterior $\wilde{D}_4$ vertices. The four sections can be chosen in such a way that three of the four $\wilde{A}_1$'s possess a vertex that is adjacent to exactly one of them. By blowing down those three vertices, we may now blow down a single exterior $\wilde{D}_4$ vertex and finally $e_7-e_8$, giving us a new blow-down. It must be remarked that we cannot blow down a $(-1)$-curve on all four of the $\wilde{A}_1$'s because then there will be no $(-1)$-curves left, and we would be stuck with eight instead of nine. In the 2, 1, 1, 0 case, we may not blow down both of the $(-1)$-curves on the $\wilde{D}_4$, because then there will be no more $(-1)$-curves to blow down afterward. There exist exactly two $(-1)$-curves on the $\wilde{A}_1$'s that each intersect exactly one of the four original blown down curves. We must blow down those two, and then our only way of blowing down three more curves is by blowing down one of the exterior $\wilde{D}_4$ vertices whose self-intersection has been raised to $-1$, then following with $e_7-e_8$ and the last $(-1)$-curve on the $\wilde{D}_4$. Finally, the 2, 2, 0, 0 case yields no valid blow-down sequence to the plane; the $\wilde{D}_4$ component contains no $(-1)$-curves, and there are no enough $(-1)$-curves from the $\wilde{A}_1$ components to bring us up to nine.

Now let us consider blowing down three of the sections. Then there is no way to blow down nine curves without blowing down another of the original eight sections; this is because we can blow down at most three from the $\wilde{D}_4$ component, and so we would need to blow down three more from the $\wilde{A}_1$ components. But then there can only be at most two $(-1)$-curves to be found among the $\wilde{A}_1$ components, bringing our count up to eight, without a ninth $(-1)$-curve to complete the blow-down to the plane. Therefore we must start by blowing down at least four of the sections, and so our catalogue of four blow-downs is complete.

The four ways to blow down this diagram (up to isomorphism) are given by the following diagrams. 
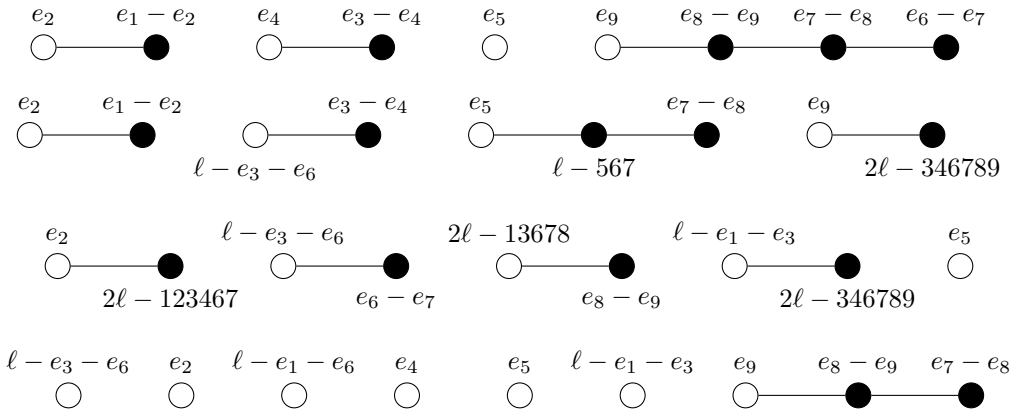
\begin{center}
\begin{tikzpicture}
    \node[shape=circle,draw=black,fill=none, label=$e_2$] (A) at (0,0) {} ;
    \node[shape=circle,draw=black,fill=black, label=$e_1-e_2$] (B) at (1.5,0) {} ;
    \node[shape=circle,draw=black,fill=none,label=$e_4$] (C) at (3,0) {} ;
    \node[shape=circle,draw=black,fill=black, label=$e_3-e_4$] (D) at (4.5,0) {} ;
    \node[shape=circle,draw=black,fill=none, label=$e_5$] (E) at (6,0) {} ;
    \node[shape=circle,draw=black,fill=none, label=$e_9$] (F) at (7.5,0) {} ;
    \node[shape=circle,draw=black,fill=black, label=$e_8-e_9$] (G) at (9,0) {} ;
    \node[shape=circle,draw=black,fill=black, label=$e_7-e_8$] (H) at (10.5,0) {} ;
    \node[shape=circle,draw=black,fill=black, label=$e_6-e_7$] (I) at (12,0) {} ;

    \path [-,draw=black](A) edge node[left] {} (B);
    \path [-,draw=black](C) edge node[left] {} (D);
    \path [-,draw=black](F) edge node[left] {} (G);
    \path [-,draw=black](G) edge node[left] {} (H);
    \path [-,draw=black](H) edge node[left] {} (I);
\end{tikzpicture}
\end{center}

\begin{center}
\begin{tikzpicture}
    \node[shape=circle,draw=black,fill=none, label=$e_9$] (A) at (10.5,0) {} ;
    \node[shape=circle,draw=black,fill=black, label=below:$2\ell-346789$] (B) at (12,0) {} ;
    \node[shape=circle,draw=black,fill=none, label=below:$\ell-e_3-e_6$] (C) at (3,0) {} ;
    \node[shape=circle,draw=black,fill=black, label=$e_3-e_4$] (D) at (4.5,0) {} ;
    \node[shape=circle,draw=black,fill=none, label=$e_2$] (E) at (0,0) {} ;
    \node[shape=circle,draw=black,fill=black, label=$e_1-e_2$] (F) at (1.5,0) {} ;
    \node[shape=circle,draw=black,fill=none, label=$e_5$] (G) at (6,0) {} ;
    \node[shape=circle,draw=black,fill=black, label=below:$\ell-567$] (H) at (7.5,0) {} ;
    \node[shape=circle,draw=black,fill=black, label=$e_7-e_8$] (I) at (9,0) {} ;

    \path [-,draw=black](A) edge node[left] {} (B);
    \path [-,draw=black](C) edge node[left] {} (D);
    \path [-,draw=black](E) edge node[left] {} (F);
    \path [-,draw=black](G) edge node[left] {} (H);
    \path [-,draw=black](H) edge node[left] {} (I);
\end{tikzpicture}
\end{center}

\begin{center}
\begin{tikzpicture}
    \node[shape=circle,draw=black,fill=none, label=$e_2$] (A) at (0,0) {} ;
    \node[shape=circle,draw=black,fill=black, label=below:$2\ell-123467$] (B) at (1.5,0) {} ;
    \node[shape=circle,draw=black,fill=none, label=$\ell-e_3-e_6$] (C) at (3,0) {} ;
    \node[shape=circle,draw=black,fill=black, label=below:$e_6-e_7$] (D) at (4.5,0) {} ;
    \node[shape=circle,draw=black,fill=none, label=$2\ell-13678$] (E) at (6,0) {} ;
    \node[shape=circle,draw=black,fill=black, label=below:$e_8-e_9$] (F) at (7.5,0) {} ;
    \node[shape=circle,draw=black,fill=none, label=$\ell-e_1-e_3$] (G) at (9,0) {} ;
    \node[shape=circle,draw=black,fill=black, label=below:$2\ell-346789$] (H) at (10.5,0) {} ;
    \node[shape=circle,draw=black,fill=none, label=$e_5$] (I) at (12,0) {} ;

    \path [-,draw=black](A) edge node[left] {} (B);
    \path [-,draw=black](C) edge node[left] {} (D);
    \path [-,draw=black](E) edge node[left] {} (F);
    \path [-,draw=black](G) edge node[left] {} (H);
\end{tikzpicture}
\end{center}
\begin{center}
\begin{tikzpicture}
    \node[shape=circle,draw=black,fill=none, label=$e_9$] (A) at (9,0) {} ;
    \node[shape=circle,draw=black,fill=black, label=$e_8-e_9$] (B) at (10.5,0) {} ;
    \node[shape=circle,draw=black,fill=black, label=$e_7-e_8$] (C) at (12,0) {} ;
    \node[shape=circle,draw=black,fill=none, label=$\ell-e_3-e_6$] (D) at (0,0) {} ;
    \node[shape=circle,draw=black,fill=none, label=$\ell-e_1-e_6$] (E) at (3,0) {} ;
    \node[shape=circle,draw=black,fill=none, label=$e_2$] (F) at (1.5,0) {} ;
    \node[shape=circle,draw=black,fill=none, label=$e_4$] (G) at (4.5,0) {} ;
    \node[shape=circle,draw=black,fill=none, label=$\ell-e_1-e_3$] (H) at (7.5,0) {} ;
    \node[shape=circle,draw=black,fill=none, label=$e_5$] (I) at (6,0) {} ;

    \path [-,draw=black](A) edge node[left] {} (B);
    \path [-,draw=black](B) edge node[left] {} (C);
\end{tikzpicture}
\end{center}
The three bottom blow-downs give us the matrices $$\left(\begin{matrix}
3 & 0 & 0 & 1 & 1 & 0 & 2 & 1 & 1 & 0 \\
0 & 1 & 0 & 0 & 0 & 0 & 0 & 0 & 0 & 0 \\
0 & 0 & 1 & 0 & 0 & 0 & 0 & 0 & 0 & 0 \\
-1 & 0 & 0 & 0 & -1 & 0 & -1 & 0 & 0 & 0 \\
-1 & 0 & 0 & -1 & 0 & 0 & -1 & 0 & 0 & 0 \\
-1 & 0 & 0 & 0 & 0 & 0 & -1 & 0 & -1 & 0 \\
-1 & 0 & 0 & 0 & 0 & 0 & -1 & -1 & 0 & 0 \\
0 & 0 & 0 & 0 & 0 & 1 & 0 & 0 & 0 & 0 \\
-2 & 0 & 0 & -1 & -1 & 0 & -1 & -1 & -1 & 0 \\
0 & 0 & 0 & 0 & 0 & 0 & 0 & 0 & 0 & 1
\end{matrix}\right):\Pic X\to\Pic X,$$
$$\left(\begin{matrix}
5 & 2 & 0 & 3 & 1 & 0 & 2 & 2 & 1 & 1 \\
-2 & -1 & 0 & -1 & -1 & 0 & -1 & -1 & 0 & 0 \\
0 & 0 & 1 & 0 & 0 & 0 & 0 & 0 & 0 & 0 \\
-1 & 0 & 0 & -1 & 0 & 0 & 0 & -1 & 0 & 0 \\
-1 & 0 & 0 & -1 & 0 & 0 & -1 & 0 & 0 & 0 \\
-2 & -1 & 0 & -1 & 0 & 0 & -1 & -1 & 0 & -1 \\
-2 & -1 & 0 & -1 & 0 & 0 & -1 & -1 & -1 & 0 \\
-3 & -1 & 0 & -2 & -1 & 0 & -1 & -1 & -1 & -1 \\
-1 & -1 & 0 & -1 & 0 & 0 & 0 & 0 & 0 & 0 \\
0 & 0 & 0 & 0 & 0 & 1 & 0 & 0 & 0 & 0
\end{matrix}\right):\Pic X\to\Pic X,$$
and $$\left(\begin{matrix}
2 & 1 & 0 & 1 & 0 & 0 & 1 & 0 & 0 & 0 \\
-1 & 0 & 0 & -1 & 0 & 0 & -1 & 0 & 0 & 0 \\
0 & 0 & 1 & 0 & 0 & 0 & 0 & 0 & 0 & 0 \\
-1 & -1 & 0 & 0 & 0 & 0 & -1 & 0 & 0 & 0 \\
0 & 0 & 0 & 0 & 1 & 0 & 0 & 0 & 0 & 0 \\
0 & 0 & 0 & 0 & 0 & 1 & 0 & 0 & 0 & 0 \\
-1 & -1 & 0 & -1 & 0 & 0 & 0 & 0 & 0 & 0 \\
0 & 0 & 0 & 0 & 0 & 0 & 0 & 1 & 0 & 0 \\
0 & 0 & 0 & 0 & 0 & 0 & 0 & 0 & 1 & 0 \\
0 & 0 & 0 & 0 & 0 & 0 & 0 & 0 & 0 & 1
\end{matrix}\right):\Pic X\to \Pic X$$ respectively. Applying these change-of-basis matrices to the original Dynkin diagram gives us the following diagrams.
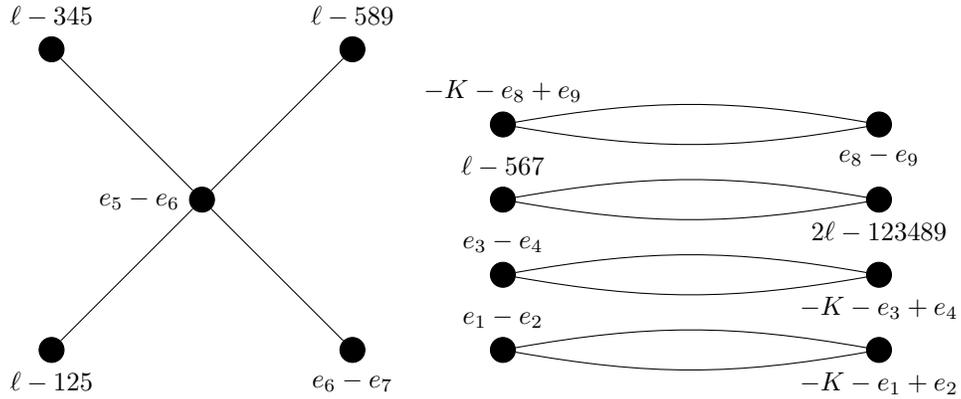
\begin{center}
    \begin{tikzpicture}
    \node[shape=circle,draw=black,fill=black,label=$\ell-345$] (A) at (0,0) {};
    \node[shape=circle,draw=black,fill=black,label=$\ell-589$] (B) at (4,0) {};
    \node[shape=circle,draw=black,fill=black,label=left:$e_5-e_6$] (C) at (2,-2) {};
    \node[shape=circle,draw=black,fill=black,label=below:$\ell-125$] (D) at (0,-4) {};
    \node[shape=circle,draw=black,fill=black,label=below:$e_6-e_7$] (E) at (4,-4) {};
    \node[shape=circle,draw=black,fill=black,label=$-K-e_8+e_9$] (F) at (6,-1) {} ;
    \node[shape=circle,draw=black,fill=black,label=below:$e_8-e_9$] (G) at (11,-1) {} ;
    \node[shape=circle,draw=black,fill=black,label=$\ell-567$] (H) at (6,-2) {} ;
    \node[shape=circle,draw=black,fill=black,label=below:$2\ell-123489$] (I) at (11,-2) {} ;
    \node[shape=circle,draw=black,fill=black,label=$e_3-e_4$] (J) at (6,-3) {} ;
    \node[shape=circle,draw=black,fill=black,label=below:$-K-e_3+e_4$] (K) at (11,-3) {} ;
    \node[shape=circle,draw=black,fill=black,label=$e_1-e_2$] (L) at (6,-4) {} ;
    \node[shape=circle,draw=black,fill=black,label=below:$-K-e_1+e_2$] (M) at (11,-4) {} ;

    \path [-](B) edge node[left] {} (C);
    \path [-](A) edge node[left] {} (C);
    \path [-](D) edge node[left] {} (C);
    \path [-](E) edge node[left] {} (C);
    \path [-,draw=black, bend left=10](F) edge node[left] {} (G);
    \path [-,draw=black, bend right=10](F) edge node[left] {} (G);
    \path [-,draw=black, bend left=10](H) edge node[left] {} (I);
    \path [-,draw=black, bend right=10](H) edge node[left] {} (I);
    \path [-,draw=black, bend left=10](J) edge node[left] {} (K);
    \path [-,draw=black, bend right=10](J) edge node[left] {} (K);
    \path [-,draw=black, bend left=10](L) edge node[left] {} (M);
    \path [-,draw=black, bend right=10](L) edge node[left] {} (M);
\end{tikzpicture}
\noindent\makebox[\linewidth]{\rule{\textwidth}{0.4pt}}
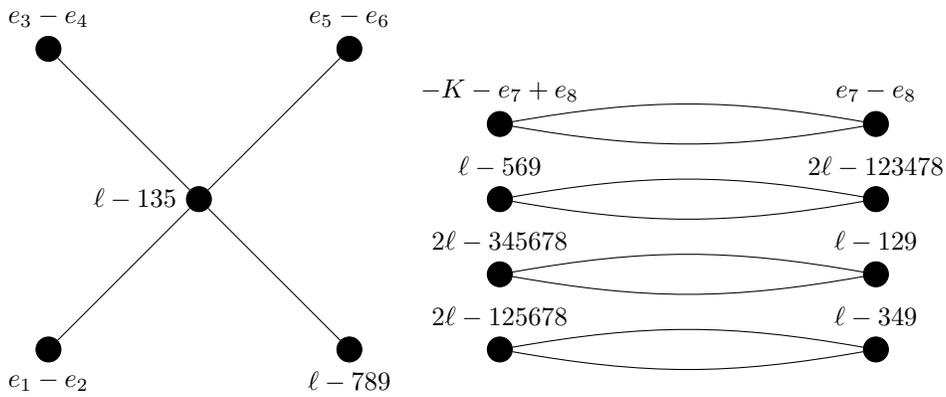
\begin{tikzpicture}
    \node[shape=circle,draw=black,fill=black,label=$e_3-e_4$] (A) at (0,0) {};
    \node[shape=circle,draw=black,fill=black,label=$e_5-e_6$] (B) at (4,0) {};
    \node[shape=circle,draw=black,fill=black,label=left:$\ell-135$] (C) at (2,-2) {};
    \node[shape=circle,draw=black,fill=black,label=below:$e_1-e_2$] (D) at (0,-4) {};
    \node[shape=circle,draw=black,fill=black,label=below:$\ell-789$] (E) at (4,-4) {};
    \node[shape=circle,draw=black,fill=black,label=$-K-e_7+e_8$] (F) at (6,-1) {} ;
    \node[shape=circle,draw=black,fill=black,label=$e_7-e_8$] (G) at (11,-1) {} ;
    \node[shape=circle,draw=black,fill=black,label=$\ell-569$] (H) at (6,-2) {} ;
    \node[shape=circle,draw=black,fill=black,label=$2\ell-123478$] (I) at (11,-2) {} ;
    \node[shape=circle,draw=black,fill=black,label=$2\ell-345678$] (J) at (6,-3) {} ;
    \node[shape=circle,draw=black,fill=black,label=$\ell-129$] (K) at (11,-3) {} ;
    \node[shape=circle,draw=black,fill=black,label=$2\ell-125678$] (L) at (6,-4) {} ;
    \node[shape=circle,draw=black,fill=black,label=$\ell-349$] (M) at (11,-4) {} ;

    \path [-](B) edge node[left] {} (C);
    \path [-](A) edge node[left] {} (C);
    \path [-](D) edge node[left] {} (C);
    \path [-](E) edge node[left] {} (C);
    \path [-,draw=black, bend left=10](F) edge node[left] {} (G);
    \path [-,draw=black, bend right=10](F) edge node[left] {} (G);
    \path [-,draw=black, bend left=10](H) edge node[left] {} (I);
    \path [-,draw=black, bend right=10](H) edge node[left] {} (I);
    \path [-,draw=black, bend left=10](J) edge node[left] {} (K);
    \path [-,draw=black, bend right=10](J) edge node[left] {} (K);
    \path [-,draw=black, bend left=10](L) edge node[left] {} (M);
    \path [-,draw=black, bend right=10](L) edge node[left] {} (M);
\end{tikzpicture}
\noindent\makebox[\linewidth]{\rule{\textwidth}{0.4pt}}
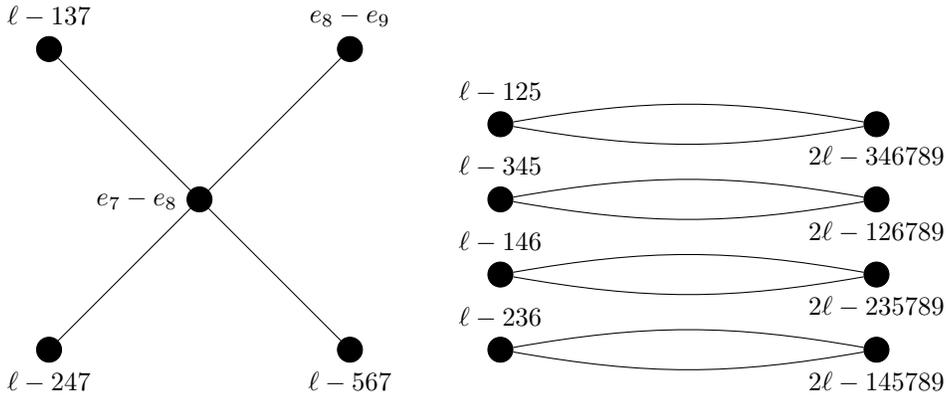
\begin{tikzpicture}
    \node[shape=circle,draw=black,fill=black,label=$\ell-137$] (A) at (0,0) {};
    \node[shape=circle,draw=black,fill=black,label=$e_8-e_9$] (B) at (4,0) {};
    \node[shape=circle,draw=black,fill=black,label=left:$e_7-e_8$] (C) at (2,-2) {};
    \node[shape=circle,draw=black,fill=black,label=below:$\ell-247$] (D) at (0,-4) {};
    \node[shape=circle,draw=black,fill=black,label=below:$\ell-567$] (E) at (4,-4) {};
    \node[shape=circle,draw=black,fill=black,label=$\ell-125$] (F) at (6,-1) {} ;
    \node[shape=circle,draw=black,fill=black,label=below:$2\ell-346789$] (G) at (11,-1) {} ;
    \node[shape=circle,draw=black,fill=black,label=$\ell-345$] (H) at (6,-2) {} ;
    \node[shape=circle,draw=black,fill=black,label=below:$2\ell-126789$] (I) at (11,-2) {} ;
    \node[shape=circle,draw=black,fill=black,label=$\ell-146$] (J) at (6,-3) {} ;
    \node[shape=circle,draw=black,fill=black,label=below:$2\ell-235789$] (K) at (11,-3) {} ;
    \node[shape=circle,draw=black,fill=black,label=$\ell-236$] (L) at (6,-4) {} ;
    \node[shape=circle,draw=black,fill=black,label=below:$2\ell-145789$] (M) at (11,-4) {} ;

    \path [-](B) edge node[left] {} (C);
    \path [-](A) edge node[left] {} (C);
    \path [-](D) edge node[left] {} (C);
    \path [-](E) edge node[left] {} (C);
    \path [-,draw=black, bend left=10](F) edge node[left] {} (G);
    \path [-,draw=black, bend right=10](F) edge node[left] {} (G);
    \path [-,draw=black, bend left=10](H) edge node[left] {} (I);
    \path [-,draw=black, bend right=10](H) edge node[left] {} (I);
    \path [-,draw=black, bend left=10](J) edge node[left] {} (K);
    \path [-,draw=black, bend right=10](J) edge node[left] {} (K);
    \path [-,draw=black, bend left=10](L) edge node[left] {} (M);
    \path [-,draw=black, bend right=10](L) edge node[left] {} (M);
\end{tikzpicture}
\end{center}
\begin{itemize}
\item The original diagram is a pencil defined by the union of a conic $C_1$ and a tangent line $L_1$, and the union of a conic $C_2$ that intersects $C_1$ at two points with multiplicity 2 and a line $L_2$ that intersects $C_1$ at one point with multiplicity 2. An example of such a pencil is $x^2z+xy^2$ and $x^2z+\vhi xz^2+\vhi^2y^2z$ where $\vhi$ satisfies $\vhi^2+\vhi+1=0$.

\item The second diagram is a pencil of the union of a conic and a line with the union of three concurrent lines, each tangent to the conic. An example of such a pencil is that spanned by $(y^2+xz)(x+\vhi z)$ and $xz(x+z)$.

\item The third diagram is a pencil spanned by the union of a conic and a line, and the union of a line that is tangent to the conic and a double line. An example of such a pencil is that spanned by $(y^2+xz)(x+\vhi z)$ and $y^2(x+z)$.

\item The fourth diagram corresponds to a pencil spanned by the union of three concurrent lines, and the union of a line and a conic that contains that contains the point where the three lines intersect. The seven ordinary points in the intersection of these two cubics form a Fano plane. An example of such a pencil is that spanned by $(x+y)(x+z)(y+z)$ and $(x+y+z)(Ax(y+z)+By(x+z))$ where $(A,B)\in\P^1$.
\end{itemize}
\begin{center}
    \textbf{Case 3:} $\wilde{A}_1^{\+2}\+\wilde{D}_6$
\end{center}
Let's now look at the Dynkin diagram $\wilde{A}_1^{\+2}\+\wilde{D}_6$, as depicted below.
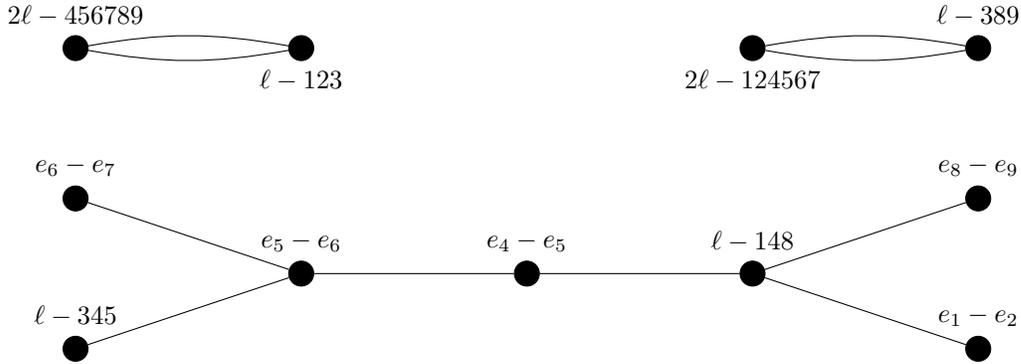
\begin{center}
\begin{adjustwidth*}{}{-0em}
\begin{tikzpicture}
    \node[shape=circle,draw=black,fill=black, label=$e_6-e_7$] (A) at (0,0) {} ;
    \node[shape=circle,draw=black,fill=black, label=$\ell-345$] (B) at (0,-2) {} ;
    \node[shape=circle,draw=black,fill=black, label=$e_5-e_6$] (C) at (3,-1) {} ;
    \node[shape=circle,draw=black,fill=black, label=$e_4-e_5$] (D) at (6,-1) {} ;
    \node[shape=circle,draw=black,fill=black, label=$\ell-148$] (E) at (9,-1) {} ;
    \node[shape=circle,draw=black,fill=black, label=$e_8-e_9$] (F) at (12,0) {} ;
    \node[shape=circle,draw=black,fill=black, label=$e_1-e_2$] (G) at (12,-2) {} ;
    \node[shape=circle,draw=black,fill=black,label=$2\ell-456789$] (H) at (0,2) {} ;
    \node[shape=circle,draw=black,fill=black, label=below:$\ell-123$] (I) at (3,2) {} ;
    \node[shape=circle,draw=black,fill=black, label=below:$2\ell-124567$] (J) at (9,2) {} ;
    \node[shape=circle,draw=black,fill=black, label=$\ell-389$] (K) at (12,2) {} ;

    \path [-,draw=black](A) edge node[left] {} (C);
    \path [-,draw=black](B) edge node[left] {} (C);
    \path [-,draw=black](C) edge node[left] {} (D);
    \path [-,draw=black](D) edge node[left] {} (E);
    \path [-,draw=black](E) edge node[left] {} (F);
    \path [-,draw=black](E) edge node[left] {} (G);
    \path [-,draw=black, bend left=10](H) edge node[left] {} (I);
    \path [-,draw=black, bend right=10](H) edge node[left] {} (I);
    \path [-,draw=black, bend left=10](J) edge node[left] {} (K);
    \path [-,draw=black, bend right=10](J) edge node[left] {} (K);
\end{tikzpicture}
\end{adjustwidth*}
\end{center}
Using Lemma \ref{MW}, we know there are four $(-1)$-curves; they are $e_2$, $e_3$, $e_7$, and $e_9$. 

By setting $\texttt{A=M2A1D6}$ and running the Macaulay2 code in the Appendix, we can see that there are-- up to isomorphism of the graph-- four ways to blow down the surface to $\P^2$ \cite{M2}. They are depicted below.

\begin{center}
\begin{tikzpicture}
    \node[shape=circle,draw=black,fill=none, label=$e_2$] (A) at (0,0) {} ;
    \node[shape=circle,draw=black,fill=black, label=$e_1-e_2$] (B) at (1.5,0) {} ;
    \node[shape=circle,draw=black,fill=none,label=$e_3$] (C) at (3,0) {} ;
    \node[shape=circle,draw=black,fill=none, label=$e_7$] (D) at (4.5,0) {} ;
    \node[shape=circle,draw=black,fill=black, label=$e_6-e_7$] (E) at (6,0) {} ;
    \node[shape=circle,draw=black,fill=black, label=$e_5-e_6$] (F) at (7.5,0) {} ;
    \node[shape=circle,draw=black,fill=black, label=$e_4-e_5$] (G) at (9,0) {} ;
    \node[shape=circle,draw=black,fill=none, label=$e_9$] (H) at (10.5,0) {} ;
    \node[shape=circle,draw=black,fill=black, label=$e_8-e_9$] (I) at (12,0) {} ;

    \path [-,draw=black](A) edge node[left] {} (B);
    \path [-,draw=black](D) edge node[left] {} (E);
    \path [-,draw=black](E) edge node[left] {} (F);
    \path [-,draw=black](F) edge node[left] {} (G);
    \path [-,draw=black](H) edge node[left] {} (I);
\end{tikzpicture}
\end{center}

\begin{center}
\begin{tikzpicture}
    \node[shape=circle,draw=black,fill=none, label=$e_2$] (A) at (0,0) {} ;
    \node[shape=circle,draw=black,fill=black, label=below:$e_1-e_2$] (B) at (1.5,0) {} ;
    \node[shape=circle,draw=black,fill=black, label=$\ell-148$] (C) at (3,0) {} ;
    \node[shape=circle,draw=black,fill=black, label=below:$e_8-e_9$] (D) at (4.5,0) {} ;
    \node[shape=circle,draw=black,fill=none, label=$e_3$] (E) at (6,0) {} ;
    \node[shape=circle,draw=black,fill=black, label=below:$\ell-345$] (F) at (7.5,0) {} ;
    \node[shape=circle,draw=black,fill=black, label=$e_5-e_6$] (G) at (9,0) {} ;
    \node[shape=circle,draw=black,fill=none, label=below:$e_7$] (H) at (10.5,0) {} ;
    \node[shape=circle,draw=black,fill=black, label=$2\ell-456789$] (I) at (12,0) {} ;

    \path [-,draw=black](A) edge node[left] {} (B);
    \path [-,draw=black](B) edge node[left] {} (C);
    \path [-,draw=black](C) edge node[left] {} (D);
    \path [-,draw=black](E) edge node[left] {} (F);
    \path [-,draw=black](F) edge node[left] {} (G);
    \path [-,draw=black](H) edge node[left] {} (I);
\end{tikzpicture}
\end{center}

\begin{center}
\begin{tikzpicture}
    \node[shape=circle,draw=black,fill=none, label=$e_2$] (A) at (0,0) {} ;
    \node[shape=circle,draw=black,fill=black, label=below:$\ell-123$] (B) at (1.5,0) {} ;
    \node[shape=circle,draw=black,fill=none, label=$e_7$] (C) at (3,0) {} ;
    \node[shape=circle,draw=black,fill=black, label=below:$e_6-e_7$] (D) at (4.5,0) {} ;
    \node[shape=circle,draw=black,fill=black, label=$e_5-e_6$] (E) at (6,0) {} ;
    \node[shape=circle,draw=black,fill=black, label=below:$e_4-e_5$] (F) at (7.5,0) {} ;
    \node[shape=circle,draw=black,fill=black, label=$\ell-148$] (G) at (9,0) {} ;
    \node[shape=circle,draw=black,fill=none, label=below:$e_9$] (H) at (10.5,0) {} ;
    \node[shape=circle,draw=black,fill=black, label=$\ell-389$] (I) at (12,0) {} ;

    \path [-,draw=black](A) edge node[left] {} (B);
    \path [-,draw=black](C) edge node[left] {} (D);
    \path [-,draw=black](D) edge node[left] {} (E);
    \path [-,draw=black](E) edge node[left] {} (F);
    \path [-,draw=black](F) edge node[left] {} (G);
    \path [-,draw=black](H) edge node[left] {} (I);
\end{tikzpicture}
\end{center}
\begin{center}
\begin{tikzpicture}
    \node[shape=circle,draw=black,fill=none, label=$e_3$] (A) at (0,0) {} ;
    \node[shape=circle,draw=black,fill=black, label=below:$\ell-123$] (B) at (1.5,0) {} ;
    \node[shape=circle,draw=black,fill=none, label=$e_7$] (C) at (3,0) {} ;
    \node[shape=circle,draw=black,fill=black, label=below:$e_6-e_7$] (D) at (4.5,0) {} ;
    \node[shape=circle,draw=black,fill=black, label=$e_5-e_6$] (E) at (6,0) {} ;
    \node[shape=circle,draw=black,fill=black, label=$e_4-e_5$] (F) at (7.5,0) {} ;
    \node[shape=circle,draw=black,fill=black, label=below:$\ell-148$] (G) at (9,0) {} ;
    \node[shape=circle,draw=black,fill=black, label=$e_1-e_2$] (H) at (10.5,0) {} ;
    \node[shape=circle,draw=black,fill=none, label=$e_9$] (I) at (12,0) {} ;

    \path [-,draw=black](A) edge node[left] {} (B);
    \path [-,draw=black](C) edge node[left] {} (D);
    \path [-,draw=black](D) edge node[left] {} (E);
    \path [-,draw=black](E) edge node[left] {} (F);
    \path [-,draw=black](F) edge node[left] {} (G);
    \path [-,draw=black](G) edge node[left] {} (H);
\end{tikzpicture}
\end{center}
The bottom three blow-downs give us the change-of-basis matrices $$\left(\begin{matrix}
3 & 0 & 0 & 0 & 2 & 1 & 1 & 0 & 1 & 1 \\
-1 & 0 & 0 & 0 & -1 & 0 & 0 & 0 & 0 & -1 \\
-1 & 0 & 0 & 0 & -1 & 0 & 0 & 0 & -1 & 0 \\
0 & 1 & 0 & 0 & 0 & 0 & 0 & 0 & 0 & 0 \\
0 & 0 & 1 & 0 & 0 & 0 & 0 & 0 & 0 & 0 \\
-1 & 0 & 0 & 0 & -1 & 0 & -1 & 0 & 0 & 0 \\
-1 & 0 & 0 & 0 & -1 & -1 & 0 & 0 & 0 & 0 \\
0 & 0 & 0 & 1 & 0 & 0 & 0 & 0 & 0 & 0 \\
-2 & 0 & 0 & 0 & -1 & -1 & -1 & 0 & -1 & -1 \\
0 & 0 & 0 & 0 & 0 & 0 & 0 & 1 & 0 & 0
\end{matrix}\right):\Pic X\to\Pic X,$$ $$\left(\begin{matrix}
2 & 1 & 0 & 1 & 0 & 0 & 0 & 0 & 1 & 0 \\
-1 & -1 & 0 & -1 & 0 & 0 & 0 & 0 & 0 & 0 \\
0 & 0 & 1 & 0 & 0 & 0 & 0 & 0 & 0 & 0 \\
-1 & -1 & 0 & 0 & 0 & 0 & 0 & 0 & -1 & 0 \\
0 & 0 & 0 & 0 & 1 & 0 & 0 & 0 & 0 & 0 \\
0 & 0 & 0 & 0 & 0 & 1 & 0 & 0 & 0 & 0 \\
0 & 0 & 0 & 0 & 0 & 0 & 1 & 0 & 0 & 0 \\
0 & 0 & 0 & 0 & 0 & 0 & 0 & 1 & 0 & 0 \\
-1 & 0 & 0 & -1 & 0 & 0 & 0 & 0 & -1 & 0 \\
0 & 0 & 0 & 0 & 0 & 0 & 0 & 0 & 0 & 1
\end{matrix}\right):\Pic X\to\Pic X,$$ and $$\left(\begin{matrix}
2 & 1 & 1 & 0 & 0 & 0 & 0 & 0 & 1 & 0 \\
-1 & -1 & -1 & 0 & 0 & 0 & 0 & 0 & 0 & 0 \\
0 & 0 & 0 & 1 & 0 & 0 & 0 & 0 & 0 & 0 \\
-1 & 0 & -1 & 0 & 0 & 0 & 0 & 0 & -1 & 0 \\
-1 & -1 & 0 & 0 & 0 & 0 & 0 & 0 & -1 & 0 \\
0 & 0 & 0 & 0 & 1 & 0 & 0 & 0 & 0 & 0 \\
0 & 0 & 0 & 0 & 0 & 1 & 0 & 0 & 0 & 0 \\
0 & 0 & 0 & 0 & 0 & 0 & 1 & 0 & 0 & 0 \\
0 & 0 & 0 & 0 & 0 & 0 & 0 & 1 & 0 & 0 \\
0 & 0 & 0 & 0 & 0 & 0 & 0 & 0 & 0 & 1
\end{matrix}\right):\Pic X\to\Pic X$$ respectively. Applying these change-of-basis matrices to the original Dynkin diagram gives us the following diagrams.
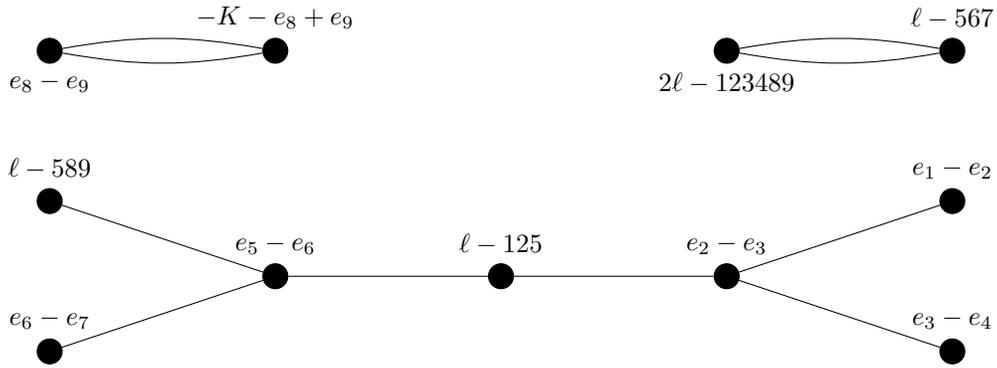
\begin{center}
\begin{tikzpicture}
    \node[shape=circle,draw=black,fill=black, label=$\ell-589$] (A) at (0,0) {} ;
    \node[shape=circle,draw=black,fill=black, label=$e_6-e_7$] (B) at (0,-2) {} ;
    \node[shape=circle,draw=black,fill=black, label=$e_5-e_6$] (C) at (3,-1) {} ;
    \node[shape=circle,draw=black,fill=black, label=$\ell-125$] (D) at (6,-1) {} ;
    \node[shape=circle,draw=black,fill=black, label=$e_2-e_3$] (E) at (9,-1) {} ;
    \node[shape=circle,draw=black,fill=black, label=$e_1-e_2$] (F) at (12,0) {} ;
    \node[shape=circle,draw=black,fill=black, label=$e_3-e_4$] (G) at (12,-2) {} ;
    \node[shape=circle,draw=black,fill=black,label=below:$e_8-e_9$] (H) at (0,2) {} ;
    \node[shape=circle,draw=black,fill=black, label=$-K-e_8+e_9$] (I) at (3,2) {} ;
    \node[shape=circle,draw=black,fill=black, label=below:$2\ell-123489$] (J) at (9,2) {} ;
    \node[shape=circle,draw=black,fill=black, label=$\ell-567$] (K) at (12,2) {} ;

    \path [-,draw=black](A) edge node[left] {} (C);
    \path [-,draw=black](B) edge node[left] {} (C);
    \path [-,draw=black](C) edge node[left] {} (D);
    \path [-,draw=black](D) edge node[left] {} (E);
    \path [-,draw=black](E) edge node[left] {} (F);
    \path [-,draw=black](E) edge node[left] {} (G);
    \path [-,draw=black, bend left=10](H) edge node[left] {} (I);
    \path [-,draw=black, bend right=10](H) edge node[left] {} (I);
    \path [-,draw=black, bend left=10](J) edge node[left] {} (K);
    \path [-,draw=black, bend right=10](J) edge node[left] {} (K);
\end{tikzpicture}
\end{center}
\noindent\makebox[\linewidth]{\rule{\textwidth}{0.4pt}}
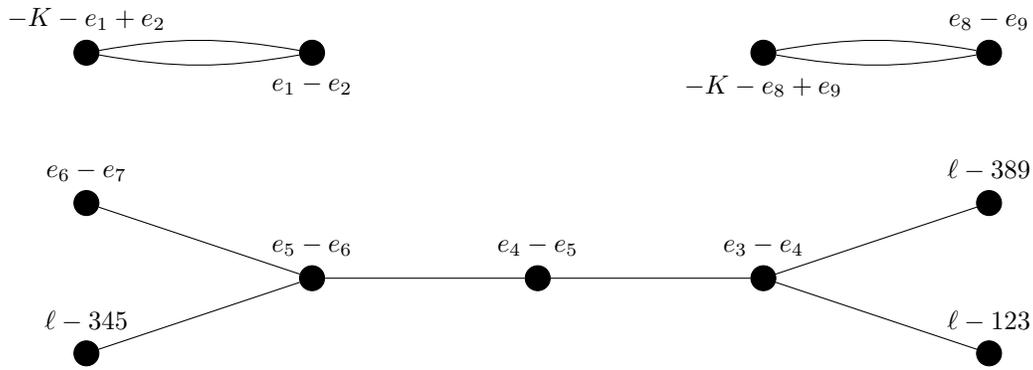
\begin{center}
\begin{tikzpicture}
    \node[shape=circle,draw=black,fill=black, label=$e_6-e_7$] (A) at (0,0) {} ;
    \node[shape=circle,draw=black,fill=black, label=$\ell-345$] (B) at (0,-2) {} ;
    \node[shape=circle,draw=black,fill=black, label=$e_5-e_6$] (C) at (3,-1) {} ;
    \node[shape=circle,draw=black,fill=black, label=$e_4-e_5$] (D) at (6,-1) {} ;
    \node[shape=circle,draw=black,fill=black, label=$e_3-e_4$] (E) at (9,-1) {} ;
    \node[shape=circle,draw=black,fill=black, label=$\ell-389$] (F) at (12,0) {} ;
    \node[shape=circle,draw=black,fill=black, label=$\ell-123$] (G) at (12,-2) {} ;
    \node[shape=circle,draw=black,fill=black,label=$-K-e_1+e_2$] (H) at (0,2) {} ;
    \node[shape=circle,draw=black,fill=black, label=below:$e_1-e_2$] (I) at (3,2) {} ;
    \node[shape=circle,draw=black,fill=black, label=below:$-K-e_8+e_9$] (J) at (9,2) {} ;
    \node[shape=circle,draw=black,fill=black, label=$e_8-e_9$] (K) at (12,2) {} ;

    \path [-,draw=black](A) edge node[left] {} (C);
    \path [-,draw=black](B) edge node[left] {} (C);
    \path [-,draw=black](C) edge node[left] {} (D);
    \path [-,draw=black](D) edge node[left] {} (E);
    \path [-,draw=black](E) edge node[left] {} (F);
    \path [-,draw=black](E) edge node[left] {} (G);
    \path [-,draw=black, bend left=10](H) edge node[left] {} (I);
    \path [-,draw=black, bend right=10](H) edge node[left] {} (I);
    \path [-,draw=black, bend left=10](J) edge node[left] {} (K);
    \path [-,draw=black, bend right=10](J) edge node[left] {} (K);
\end{tikzpicture}
\end{center}
\noindent\makebox[\linewidth]{\rule{\textwidth}{0.4pt}}
\begin{center}

\begin{adjustwidth*}{}{-0em}
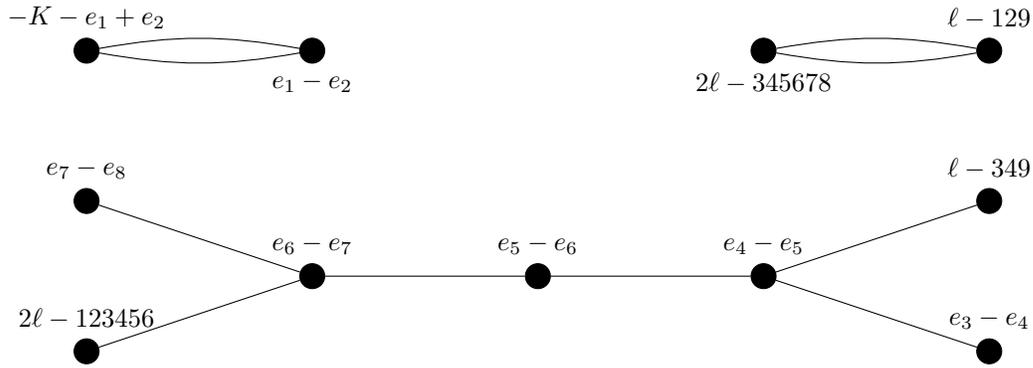
\begin{tikzpicture}
    \node[shape=circle,draw=black,fill=black, label=$e_7-e_8$] (A) at (0,0) {} ;
    \node[shape=circle,draw=black,fill=black, label=$2\ell-123456$] (B) at (0,-2) {} ;
    \node[shape=circle,draw=black,fill=black, label=$e_6-e_7$] (C) at (3,-1) {} ;
    \node[shape=circle,draw=black,fill=black, label=$e_5-e_6$] (D) at (6,-1) {} ;
    \node[shape=circle,draw=black,fill=black, label=$e_4-e_5$] (E) at (9,-1) {} ;
    \node[shape=circle,draw=black,fill=black, label=$\ell-349$] (F) at (12,0) {} ;
    \node[shape=circle,draw=black,fill=black, label=$e_3-e_4$] (G) at (12,-2) {} ;
    \node[shape=circle,draw=black,fill=black,label=$-K-e_1+e_2$] (H) at (0,2) {} ;
    \node[shape=circle,draw=black,fill=black, label=below:$e_1-e_2$] (I) at (3,2) {} ;
    \node[shape=circle,draw=black,fill=black, label=below:$2\ell-345678$] (J) at (9,2) {} ;
    \node[shape=circle,draw=black,fill=black, label=$\ell-129$] (K) at (12,2) {} ;

    \path [-,draw=black](A) edge node[left] {} (C);
    \path [-,draw=black](B) edge node[left] {} (C);
    \path [-,draw=black](C) edge node[left] {} (D);
    \path [-,draw=black](D) edge node[left] {} (E);
    \path [-,draw=black](E) edge node[left] {} (F);
    \path [-,draw=black](E) edge node[left] {} (G);
    \path [-,draw=black, bend left=10](H) edge node[left] {} (I);
    \path [-,draw=black, bend right=10](H) edge node[left] {} (I);
    \path [-,draw=black, bend left=10](J) edge node[left] {} (K);
    \path [-,draw=black, bend right=10](J) edge node[left] {} (K);
\end{tikzpicture}
\end{adjustwidth*}
\end{center}
\begin{itemize}
\item The original diagram corresponds to a pencil spanned by the union of a conic and a line, and the union of a line tangent to the conic and a double line. An example of such a pencil is that spanned by $(y^2+xz)(x+z)$ and $xy^2$.

\item The second diagram corresponds to a pencil spanned by the union of a conic and a line, and the union of two lines tangent to the conic and intersect on the first line, one of which is double. An example of such a pencil is $(y^2+xz)(x+z)$ and $x^2z$.

\item The third diagram corresponds to a pencil spanned by the union of three concurrent lines, and a cuspidal cubic that has a cusp on one of the lines, is tangent to another line at the concurrent intersection point, and is tangent to the third line elsewhere. An example of such a pencil is that spanned by $x^3+y^2z$ and $xz(x+z)$.

\item The fourth diagram corresponds to a pencil spanned by the union of a conic and a tangent line, and the union of a line that is tangent to the first conic and a conic that intersects the first conic with multiplicity 4 at the original tangent point, creating a multiplicity-6 intersection of the two cubics at that point. An example of such a pencil is that spanned by $(y^2+xz)x$ and $z(y^2+x^2+xz)$. The modulus is 1 because there is a pencil of conics that meet each other at a point of multiplicity 4.
\end{itemize}
\begin{center}
    \textbf{Case 4:} $\wilde{D}_4^{\+2}$
\end{center}
Let's now look at the Dynkin diagram $\wilde{D}_4^{\+2}$, as depicted below.
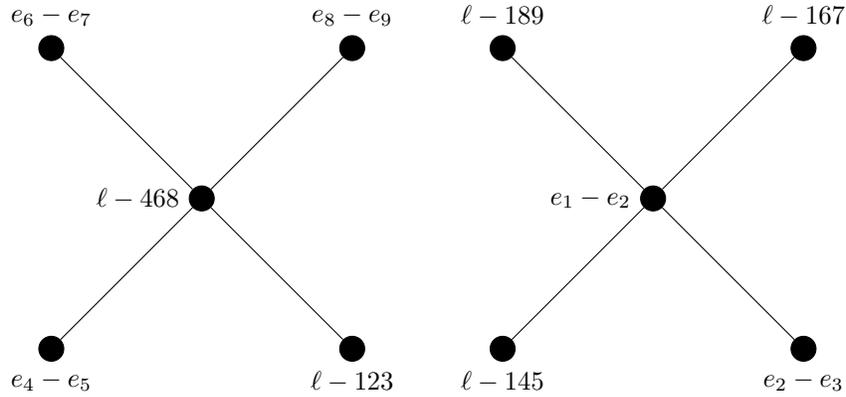
\begin{center}
\begin{tikzpicture}
    \node[shape=circle,draw=black,fill=black,label=$e_6-e_7$] (A) at (0,0) {};
    \node[shape=circle,draw=black,fill=black,label=$e_8-e_9$] (B) at (4,0) {};
    \node[shape=circle,draw=black,fill=black,label=left:$\ell-468$] (C) at (2,-2) {};
    \node[shape=circle,draw=black,fill=black,label=below:$e_4-e_5$] (D) at (0,-4) {};
    \node[shape=circle,draw=black,fill=black,label=below:$\ell-123$] (E) at (4,-4) {};
    
    \node[shape=circle,draw=black,fill=black,label=$\ell-189$] (F) at (6,0) {};
    \node[shape=circle,draw=black,fill=black,label=$\ell-167$] (G) at (10,0) {};
    \node[shape=circle,draw=black,fill=black,label=left:$e_1-e_2$] (H) at (8,-2) {};
    \node[shape=circle,draw=black,fill=black,label=below:$\ell-145$] (I) at (6,-4) {};
    \node[shape=circle,draw=black,fill=black,label=below:$e_2-e_3$] (J) at (10,-4) {};

    \path [-](B) edge node[left] {} (C);
    \path [-](A) edge node[left] {} (C);
    \path [-](D) edge node[left] {} (C);
    \path [-](E) edge node[left] {} (C);
    
    \path [-](F) edge node[left] {} (H);
    \path [-](G) edge node[left] {} (H);
    \path [-](I) edge node[left] {} (H);
    \path [-](J) edge node[left] {} (H);
\end{tikzpicture}
\end{center}
Using Lemma \ref{MW}, we know there are four $(-1)$-curves of this diagram; they are $e_3$, $e_5$, $e_7$, and $e_9$. 

By setting $\texttt{A=M2D4}$ and running the Macaulay2 code in the Appendix, we can see that there are-- up to isomorphism of the graph-- two ways to blow down the surface to $\P^2$. They are depicted below.

\begin{center}
\begin{tikzpicture}
    \node[shape=circle,draw=black,fill=none, label=$e_3$] (A) at (0,0) {} ;
    \node[shape=circle,draw=black,fill=black, label=$e_2-e_3$] (B) at (1.5,0) {} ;
    \node[shape=circle,draw=black,fill=black,label=$e_1-e_2$] (C) at (3,0) {} ;
    \node[shape=circle,draw=black,fill=none, label=$e_5$] (D) at (4.5,0) {} ;
    \node[shape=circle,draw=black,fill=black, label=$e_4-e_5$] (E) at (6,0) {} ;
    \node[shape=circle,draw=black,fill=none, label=$e_7$] (F) at (7.5,0) {} ;
    \node[shape=circle,draw=black,fill=black, label=$e_6-e_7$] (G) at (9,0) {} ;
    \node[shape=circle,draw=black,fill=none, label=$e_9$] (H) at (10.5,0) {} ;
    \node[shape=circle,draw=black,fill=black, label=$e_8-e_9$] (I) at (12,0) {} ;

    \path [-,draw=black](A) edge node[left] {} (B);
    \path [-,draw=black](B) edge node[left] {} (C);
    \path [-,draw=black](D) edge node[left] {} (E);
    \path [-,draw=black](F) edge node[left] {} (G);
    \path [-,draw=black](H) edge node[left] {} (I);
\end{tikzpicture}
\end{center}

\begin{center}
\begin{tikzpicture}
    \node[shape=circle,draw=black,fill=none, label=$e_3$] (A) at (0,0) {} ;
    \node[shape=circle,draw=black,fill=black, label=below:$e_2-e_3$] (B) at (1.5,0) {} ;
    \node[shape=circle,draw=black,fill=black, label=$e_1-e_2$] (C) at (3,0) {} ;
    \node[shape=circle,draw=black,fill=black, label=below:$\ell-145$] (D) at (4.5,0) {} ;
    \node[shape=circle,draw=black,fill=none, label=$e_7$] (E) at (6,0) {} ;
    \node[shape=circle,draw=black,fill=black, label=below:$e_6-e_7$] (F) at (7.5,0) {} ;
    \node[shape=circle,draw=black,fill=black, label=$\ell-468$] (G) at (9,0) {} ;
    \node[shape=circle,draw=black,fill=black, label=below:$e_4-e_5$] (H) at (10.5,0) {} ;
    \node[shape=circle,draw=black,fill=none, label=$e_9$] (I) at (12,0) {} ;

    \path [-,draw=black](A) edge node[left] {} (B);
    \path [-,draw=black](B) edge node[left] {} (C);
    \path [-,draw=black](C) edge node[left] {} (D);
    \path [-,draw=black](E) edge node[left] {} (F);
    \path [-,draw=black](F) edge node[left] {} (G);
    \path [-,draw=black](G) edge node[left] {} (H);
\end{tikzpicture}
\end{center}
The second blow-down gives us the following change-of-basis matrix. $$\left(\begin{matrix}
2 & 0 & 0 & 0 & 1 & 1 & 0 & 0 & 1 & 0 \\
-1 & 0 & 0 & 0 & -1 & -1 & 0 & 0 & 0 & 0 \\
0 & 1 & 0 & 0 & 0 & 0 & 0 & 0 & 0 & 0 \\
0 & 0 & 1 & 0 & 0 & 0 & 0 & 0 & 0 & 0 \\
0 & 0 & 0 & 1 & 0 & 0 & 0 & 0 & 0 & 0 \\
-1 & 0 & 0 & 0 & 0 & -1 & 0 & 0 & -1 & 0 \\
-1 & 0 & 0 & 0 & -1 & 0 & 0 & 0 & -1 & 0 \\
0 & 0 & 0 & 0 & 0 & 0 & 1 & 0 & 0 & 0 \\
0 & 0 & 0 & 0 & 0 & 0 & 0 & 1 & 0 & 0 \\
0 & 0 & 0 & 0 & 0 & 0 & 0 & 0 & 0 & 1
\end{matrix}\right):\Pic X\to\Pic X.$$
Applying this matrix to the Dynkin diagram gives us the following relabeling.
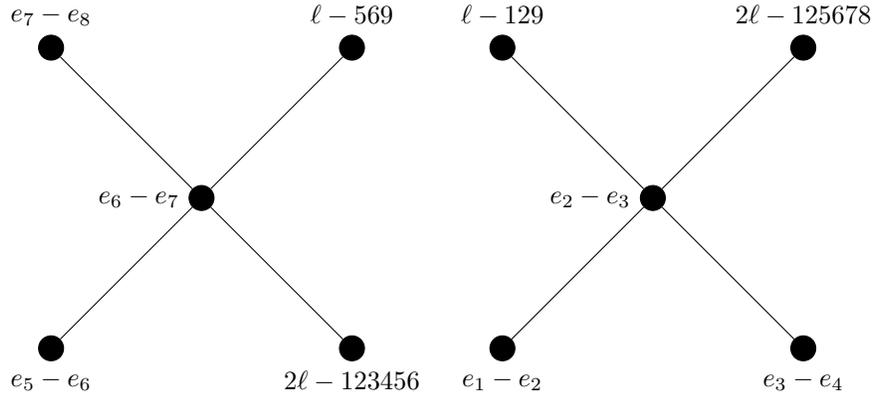
\begin{center}
\begin{tikzpicture}
    \node[shape=circle,draw=black,fill=black,label=$e_7-e_8$] (A) at (0,0) {};
    \node[shape=circle,draw=black,fill=black,label=$\ell-569$] (B) at (4,0) {};
    \node[shape=circle,draw=black,fill=black,label=left:$e_6-e_7$] (C) at (2,-2) {};
    \node[shape=circle,draw=black,fill=black,label=below:$e_5-e_6$] (D) at (0,-4) {};
    \node[shape=circle,draw=black,fill=black,label=below:$2\ell-123456$] (E) at (4,-4) {};
    
    \node[shape=circle,draw=black,fill=black,label=$\ell-129$] (F) at (6,0) {};
    \node[shape=circle,draw=black,fill=black,label=$2\ell-125678$] (G) at (10,0) {};
    \node[shape=circle,draw=black,fill=black,label=left:$e_2-e_3$] (H) at (8,-2) {};
    \node[shape=circle,draw=black,fill=black,label=below:$e_1-e_2$] (I) at (6,-4) {};
    \node[shape=circle,draw=black,fill=black,label=below:$e_3-e_4$] (J) at (10,-4) {};

    \path [-](B) edge node[left] {} (C);
    \path [-](A) edge node[left] {} (C);
    \path [-](D) edge node[left] {} (C);
    \path [-](E) edge node[left] {} (C);
    
    \path [-](F) edge node[left] {} (H);
    \path [-](G) edge node[left] {} (H);
    \path [-](I) edge node[left] {} (H);
    \path [-](J) edge node[left] {} (H);
\end{tikzpicture}
\end{center}
\begin{itemize}
\item The original diagram corresponds to a pencil spanned by the union of a line and a double line, and the union of three concurrent lines which all intersect on the single line. An example of such a pencil is that spanned by $x(x+z)z$ and $y^2(x+\vhi z)$.

\item The second diagram corresponds to a pencil spanned by the union of a conic $C_1$ and a tangent line $L_1$, and the union of a conic $C_2$ which intersects $C_1$ at two points with multiplicity 2, one of which is the intersection point of $C_1$ with $L_1$, and a line $L_2$ which is tangent to $C_1$ and $C_2$ at the other point where they meet with multiplicity 2. An example of such a pencil is that spanned by $x(y^2+xz)$ and $z(\vhi y^2+xz)$.
\end{itemize}
\begin{center}
    \textbf{Case 5:} $\wilde{A}_1\+\wilde{E}_7$
\end{center}
Let's now look at the Dynkin diagram $\wilde{A}_1\+\wilde{E}_7$, as depicted below.
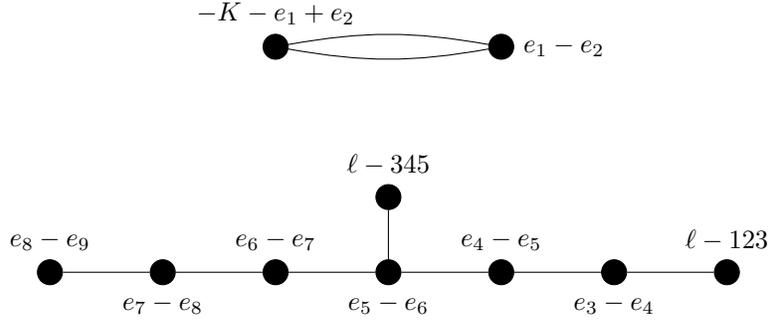
\begin{center}
\begin{tikzpicture}
    \node[shape=circle,draw=black,fill=black, label=$e_8-e_9$] (A) at (0,0) {} ;
    \node[shape=circle,draw=black,fill=black, label=below:$e_7-e_8$] (B) at (1.5,0) {} ;
    \node[shape=circle,draw=black,fill=black, label=$e_6-e_7$] (C) at (3,0) {} ;
    \node[shape=circle,draw=black,fill=black, label=below:$e_5-e_6$] (D) at (4.5,0) {} ;
    \node[shape=circle,draw=black,fill=black, label=$e_4-e_5$] (E) at (6,0) {} ;
    \node[shape=circle,draw=black,fill=black, label=below:$e_3-e_4$] (F) at (7.5,0) {} ;
    \node[shape=circle,draw=black,fill=black, label=$\ell-123$] (G) at (9,0) {} ;
    \node[shape=circle,draw=black,fill=black, label=$\ell-345$] (H) at (4.5,1) {} ;
    \node[shape=circle,draw=black,fill=black, label=$-K-e_1+e_2$] (I) at (3,3) {} ;
    \node[shape=circle,draw=black,fill=black, label=right:$e_1-e_2$] (J) at (6,3) {} ;

    \path [-,draw=black](A) edge node[left] {} (B);
    \path [-,draw=black](B) edge node[left] {} (C);
    \path [-,draw=black](C) edge node[left] {} (D);
    \path [-,draw=black](E) edge node[left] {} (F);
    \path [-,draw=black](F) edge node[left] {} (G);
    \path [-,draw=black](D) edge node[left] {} (E);
    \path [-,draw=black](D) edge node[left] {} (H);
    \path [-,draw=black, bend left=10](I) edge node[left] {} (J);
    \path [-,draw=black, bend right=10](I) edge node[left] {} (J);
\end{tikzpicture}
\end{center}
Using Lemma \ref{MW}, we know there are two $(-1)$-curves of this diagram; they are $e_2$ and $e_9$. 

By setting $\texttt{A=M2A1D6}$ and running the Macaulay2 code in the Appendix, we can see that there are-- up to isomorphism of the graph-- two ways to blow down the surface to $\P^2$. They are depicted below.

\begin{center}
\begin{tikzpicture}
    \node[shape=circle,draw=black,fill=none, label=$e_2$] (A) at (0,0) {} ;
    \node[shape=circle,draw=black,fill=black, label=$e_1-e_2$] (B) at (1.5,0) {} ;
    \node[shape=circle,draw=black,fill=none,label=$e_9$] (C) at (3,0) {} ;
    \node[shape=circle,draw=black,fill=black, label=$e_8-e_9$] (D) at (4.5,0) {} ;
    \node[shape=circle,draw=black,fill=black, label=$e_7-e_8$] (E) at (6,0) {} ;
    \node[shape=circle,draw=black,fill=black, label=$e_6-e_7$] (F) at (7.5,0) {} ;
    \node[shape=circle,draw=black,fill=black, label=$e_5-e_6$] (G) at (9,0) {} ;
    \node[shape=circle,draw=black,fill=black, label=$e_4-e_5$] (H) at (10.5,0) {} ;
    \node[shape=circle,draw=black,fill=black, label=$e_3-e_4$] (I) at (12,0) {} ;

    \path [-,draw=black](A) edge node[left] {} (B);
    \path [-,draw=black](D) edge node[left] {} (C);
    \path [-,draw=black](D) edge node[left] {} (E);
    \path [-,draw=black](F) edge node[left] {} (G);
    \path [-,draw=black](H) edge node[left] {} (I);
    \path [-,draw=black](E) edge node[left] {} (F);
    \path [-,draw=black](G) edge node[left] {} (H);
\end{tikzpicture}
\end{center}

\begin{center}
\begin{tikzpicture}
    \node[shape=circle,draw=black,fill=none, label=$e_2$] (A) at (0,0) {} ;
    \node[shape=circle,draw=black,fill=black, label=below:$\ell-123$] (B) at (1.5,0) {} ;
    \node[shape=circle,draw=black,fill=black, label=$e_3-e_4$] (C) at (3,0) {} ;
    \node[shape=circle,draw=black,fill=none, label=below:$e_9$] (D) at (4.5,0) {} ;
    \node[shape=circle,draw=black,fill=black, label=$e_8-e_9$] (E) at (6,0) {} ;
    \node[shape=circle,draw=black,fill=black, label=below:$e_7-e_8$] (F) at (7.5,0) {} ;
    \node[shape=circle,draw=black,fill=black, label=$e_6-e_7$] (G) at (9,0) {} ;
    \node[shape=circle,draw=black,fill=black, label=below:$e_5-e_6$] (H) at (10.5,0) {} ;
    \node[shape=circle,draw=black,fill=black, label=$\ell-345$] (I) at (12,0) {} ;

    \path [-,draw=black](A) edge node[left] {} (B);
    \path [-,draw=black](B) edge node[left] {} (C);
    \path [-,draw=black](H) edge node[left] {} (I);
    \path [-,draw=black](E) edge node[left] {} (F);
    \path [-,draw=black](F) edge node[left] {} (G);
    \path [-,draw=black](G) edge node[left] {} (H);
    \path [-,draw=black](D) edge node[left] {} (E);
\end{tikzpicture}
\end{center}
The second blow-down gives us the following change-of-basis matrix. $$\left(\begin{matrix}
2 & 1 & 0 & 1 & 1 & 0 & 0 & 0 & 0 & 0 \\
-1 & -1 & 0 & 0 & -1 & 0 & 0 & 0 & 0 & 0 \\
-1 & -1 & 0 & -1 & 0 & 0 & 0 & 0 & 0 & 0 \\
0 & 0 & 1 & 0 & 0 & 0 & 0 & 0 & 0 & 0 \\
-1 & 0 & 0 & -1 & -1 & 0 & 0 & 0 & 0 & 0 \\
0 & 0 & 0 & 0 & 0 & 1 & 0 & 0 & 0 & 0 \\
0 & 0 & 0 & 0 & 0 & 0 & 1 & 0 & 0 & 0 \\
0 & 0 & 0 & 0 & 0 & 0 & 0 & 1 & 0 & 0 \\
0 & 0 & 0 & 0 & 0 & 0 & 0 & 0 & 1 & 0 \\
0 & 0 & 0 & 0 & 0 & 0 & 0 & 0 & 0 & 1
\end{matrix}\right):\Pic X\to\Pic X.$$ Applying this matrix to the Dynkin diagram gives us the following relabeling.
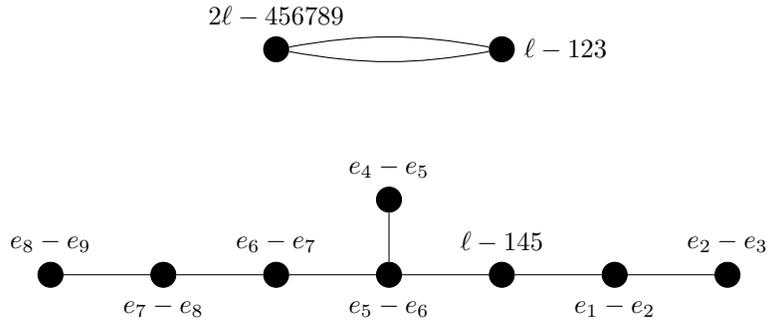
\begin{center}
\begin{tikzpicture}
    \node[shape=circle,draw=black,fill=black, label=$e_8-e_9$] (A) at (0,0) {} ;
    \node[shape=circle,draw=black,fill=black, label=below:$e_7-e_8$] (B) at (1.5,0) {} ;
    \node[shape=circle,draw=black,fill=black, label=$e_6-e_7$] (C) at (3,0) {} ;
    \node[shape=circle,draw=black,fill=black, label=below:$e_5-e_6$] (D) at (4.5,0) {} ;
    \node[shape=circle,draw=black,fill=black, label=$\ell-145$] (E) at (6,0) {} ;
    \node[shape=circle,draw=black,fill=black, label=below:$e_1-e_2$] (F) at (7.5,0) {} ;
    \node[shape=circle,draw=black,fill=black, label=$e_2-e_3$] (G) at (9,0) {} ;
    \node[shape=circle,draw=black,fill=black, label=$e_4-e_5$] (H) at (4.5,1) {} ;
    \node[shape=circle,draw=black,fill=black, label=$2\ell-456789$] (I) at (3,3) {} ;
    \node[shape=circle,draw=black,fill=black, label=right:$\ell-123$] (J) at (6,3) {} ;

    \path [-,draw=black](A) edge node[left] {} (B);
    \path [-,draw=black](B) edge node[left] {} (C);
    \path [-,draw=black](C) edge node[left] {} (D);
    \path [-,draw=black](E) edge node[left] {} (F);
    \path [-,draw=black](F) edge node[left] {} (G);
    \path [-,draw=black](D) edge node[left] {} (E);
    \path [-,draw=black](D) edge node[left] {} (H);
    \path [-,draw=black, bend left=10](I) edge node[left] {} (J);
    \path [-,draw=black, bend right=10](I) edge node[left] {} (J);
\end{tikzpicture}
\end{center}
\begin{itemize}
\item The first diagram gives us a pencil spanned by a cuspidal cubic, and the union of a line and a double line, where the double line is tangent to the flex point of the cuspidal cubic and the line goes through the cusp and the flex point. An example of such a pencil is that spanned by $x^2+y^2z$ and $xz^2$.

\item The second diagram corresponds to  a pencil spanned by the union of a conic and a line, and a triple line that is tangent to the conic. An example of such a pencil is that spanned by $(y^2+xz)z$ and $x^3$.
\end{itemize}
\begin{center}
    \textbf{Case 6:} $\wilde{D}_8$
\end{center}
Let's now look at the Dynkin diagram $\wilde{D}_8$, as depicted below.
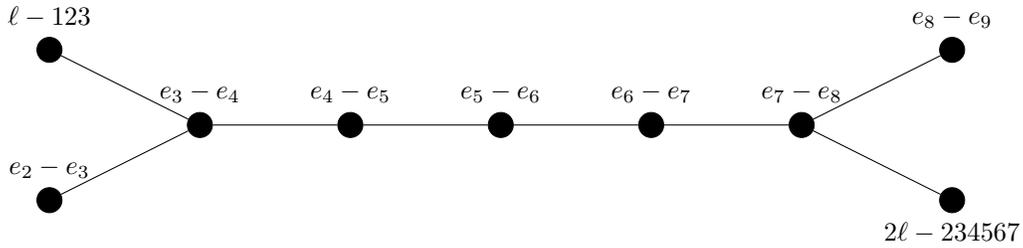
\begin{center}
\begin{tikzpicture}
    \node[shape=circle,draw=black,fill=black, label=$\ell-123$] (A) at (0,0) {} ;
    \node[shape=circle,draw=black,fill=black, label=$e_2-e_3$] (B) at (0,-2) {} ;
    \node[shape=circle,draw=black,fill=black, label=$e_3-e_4$] (C) at (2,-1) {} ;
    \node[shape=circle,draw=black,fill=black, label=$e_4-e_5$] (D) at (4,-1) {} ;
    \node[shape=circle,draw=black,fill=black, label=$e_5-e_6$] (E) at (6,-1) {} ;
    \node[shape=circle,draw=black,fill=black, label=$e_6-e_7$] (H) at (8,-1) {} ;
    \node[shape=circle,draw=black,fill=black, label=$e_7-e_8$] (I) at (10,-1) {} ;
    \node[shape=circle,draw=black,fill=black, label=$e_8-e_9$] (F) at (12,0) {} ;
    \node[shape=circle,draw=black,fill=black, label=below:$2\ell-234567$] (G) at (12,-2) {} ;

    \path [-,draw=black](A) edge node[left] {} (C);
    \path [-,draw=black](B) edge node[left] {} (C);
    \path [-,draw=black](C) edge node[left] {} (D);
    \path [-,draw=black](D) edge node[left] {} (E);
    \path [-,draw=black](E) edge node[left] {} (H);
    \path [-,draw=black](H) edge node[left] {} (I);
    \path [-,draw=black](F) edge node[left] {} (I);
    \path [-,draw=black](G) edge node[left] {} (I);
\end{tikzpicture}
\end{center}
Using Lemma \ref{MW}, we know there are two $(-1)$-curves of this diagram; they are $e_1$ and $e_9$. 

By setting $\texttt{A=MD8}$ and running the Macaulay2 code in the Appendix, we can see that there are-- up to isomorphism of the graph-- two ways to blow down the surface to $\P^2$. They are depicted below.

\begin{center}
\begin{tikzpicture}
    \node[shape=circle,draw=black,fill=none, label=$e_1$] (A) at (0,0) {} ;
    \node[shape=circle,draw=black,fill=none, label=$e_9$] (B) at (1.5,0) {} ;
    \node[shape=circle,draw=black,fill=black,label=$e_8-e_9$] (C) at (3,0) {} ;
    \node[shape=circle,draw=black,fill=black, label=$e_7-e_8$] (D) at (4.5,0) {} ;
    \node[shape=circle,draw=black,fill=black, label=$e_6-e_7$] (E) at (6,0) {} ;
    \node[shape=circle,draw=black,fill=black, label=$e_5-e_6$] (F) at (7.5,0) {} ;
    \node[shape=circle,draw=black,fill=black, label=$e_4-e_5$] (G) at (9,0) {} ;
    \node[shape=circle,draw=black,fill=black, label=$e_3-e_4$] (H) at (10.5,0) {} ;
    \node[shape=circle,draw=black,fill=black, label=$e_2-e_3$] (I) at (12,0) {} ;

    \path [-,draw=black](C) edge node[left] {} (B);
    \path [-,draw=black](D) edge node[left] {} (C);
    \path [-,draw=black](D) edge node[left] {} (E);
    \path [-,draw=black](F) edge node[left] {} (G);
    \path [-,draw=black](H) edge node[left] {} (I);
    \path [-,draw=black](E) edge node[left] {} (F);
    \path [-,draw=black](G) edge node[left] {} (H);
\end{tikzpicture}
\end{center}

\begin{center}
\begin{tikzpicture}
    \node[shape=circle,draw=black,fill=none, label=$e_1$] (A) at (0,0) {} ;
    \node[shape=circle,draw=black,fill=black, label=below:$\ell-123$] (B) at (1.5,0) {} ;
    \node[shape=circle,draw=black,fill=black, label=$e_3-e_4$] (C) at (3,0) {} ;
    \node[shape=circle,draw=black,fill=black, label=below:$e_4-e_5$] (D) at (4.5,0) {} ;
    \node[shape=circle,draw=black,fill=black, label=$e_5-e_6$] (E) at (6,0) {} ;
    \node[shape=circle,draw=black,fill=none, label=$e_9$] (F) at (7.5,0) {} ;
    \node[shape=circle,draw=black,fill=black, label=$e_8-e_9$] (G) at (9,0) {} ;
    \node[shape=circle,draw=black,fill=black, label=$e_7-e_8$] (H) at (10.5,0) {} ;
    \node[shape=circle,draw=black,fill=black, label=below:$2\ell-234567$] (I) at (12,0) {} ;

    \path [-,draw=black](A) edge node[left] {} (B);
    \path [-,draw=black](B) edge node[left] {} (C);
    \path [-,draw=black](C) edge node[left] {} (D);
    \path [-,draw=black](D) edge node[left] {} (E);
    \path [-,draw=black](F) edge node[left] {} (G);
    \path [-,draw=black](G) edge node[left] {} (H);
    \path [-,draw=black](I) edge node[left] {} (H);
\end{tikzpicture}
\end{center}
The second blow-down yields the following change-of-basis matrix. $$\left(\begin{matrix}
3 & 0 & 2 & 1 & 1 & 1 & 1 & 0 & 0 & 0 \\
-1 & 0 & -1 & 0 & 0 & 0 & -1 & 0 & 0 & 0 \\
-1 & 0 & -1 & 0 & 0 & -1 & 0 & 0 & 0 & 0 \\
-1 & 0 & -1 & 0 & -1 & 0 & 0 & 0 & 0 & 0 \\
-1 & 0 & -1 & -1 & 0 & 0 & 0 & 0 & 0 & 0 \\
0 & 1 & 0 & 0 & 0 & 0 & 0 & 0 & 0 & 0 \\
-2 & 0 & -1 & -1 & -1 & -1 & -1 & 0 & 0 & 0 \\
0 & 0 & 0 & 0 & 0 & 0 & 0 & 1 & 0 & 0 \\
0 & 0 & 0 & 0 & 0 & 0 & 0 & 0 & 1 & 0 \\
0 & 0 & 0 & 0 & 0 & 0 & 0 & 0 & 0 & 1
\end{matrix}\right):\Pic X\to\Pic X.$$ Applying the matrix to the Dynkin diagram gives us the following relabeling.
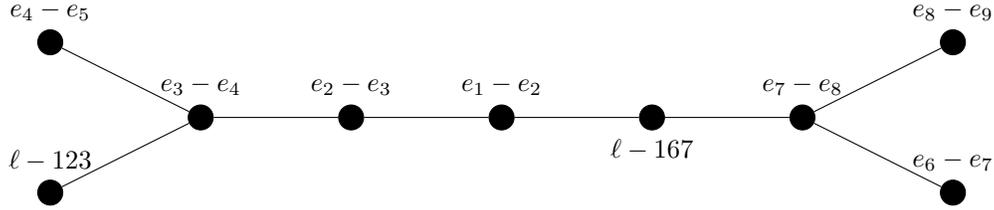
\begin{center}
\begin{tikzpicture}
    \node[shape=circle,draw=black,fill=black, label=$e_4-e_5$] (A) at (0,0) {} ;
    \node[shape=circle,draw=black,fill=black, label=$\ell-123$] (B) at (0,-2) {} ;
    \node[shape=circle,draw=black,fill=black, label=$e_3-e_4$] (C) at (2,-1) {} ;
    \node[shape=circle,draw=black,fill=black, label=$e_2-e_3$] (D) at (4,-1) {} ;
    \node[shape=circle,draw=black,fill=black, label=$e_1-e_2$] (E) at (6,-1) {} ;
    \node[shape=circle,draw=black,fill=black, label=below:$\ell-167$] (H) at (8,-1) {} ;
    \node[shape=circle,draw=black,fill=black, label=$e_7-e_8$] (I) at (10,-1) {} ;
    \node[shape=circle,draw=black,fill=black, label=$e_8-e_9$] (F) at (12,0) {} ;
    \node[shape=circle,draw=black,fill=black, label=$e_6-e_7$] (G) at (12,-2) {} ;

    \path [-,draw=black](A) edge node[left] {} (C);
    \path [-,draw=black](B) edge node[left] {} (C);
    \path [-,draw=black](C) edge node[left] {} (D);
    \path [-,draw=black](D) edge node[left] {} (E);
    \path [-,draw=black](E) edge node[left] {} (H);
    \path [-,draw=black](H) edge node[left] {} (I);
    \path [-,draw=black](F) edge node[left] {} (I);
    \path [-,draw=black](G) edge node[left] {} (I);
\end{tikzpicture}
\end{center}
\begin{itemize}
\item The first diagram corresponds to a pencil spanned by a cuspidal cubic and the union of a line tangent to the cubic at somewhere other than the cusp or flex point, and a conic that meets the cubic at the same point with multiplicity 6. (Note: in characteristic 2, all points other than the flex and cusp on a cuspidal cubic are sextactic.) An example of such a conic is that spanned by $(y^2+x^2+xz+z^2)(x+z)$ and $x^3+y^2z$.

\item The second diagram corresponds to a pencil spanned by a cuspidal cubic and the union of a line and double line, where the line is tangent to the cubic at the flex point and the double line is tangent to a sextactic point. An example of such a pencil is that spanned by $x^3+y^2z$ and $(x+z)^2z$.
\end{itemize}
\begin{center}
    \textbf{Case 7:} $\wilde{E}_8$
\end{center}
Let's now look at the Dynkin diagram $\wilde{E}_8$, as depicted below.
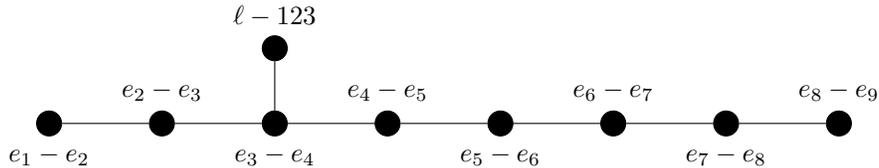
\begin{center}
\begin{tikzpicture}
    \node[shape=circle,draw=black,fill=black, label=below:$e_1-e_2$] (A) at (0,0) {} ;
    \node[shape=circle,draw=black,fill=black, label=$e_2-e_3$] (B) at (1.5,0) {} ;
    \node[shape=circle,draw=black,fill=black, label=below:$e_3-e_4$] (C) at (3,0) {} ;
    \node[shape=circle,draw=black,fill=black, label=$e_4-e_5$] (D) at (4.5,0) {} ;
    \node[shape=circle,draw=black,fill=black, label=below:$e_5-e_6$] (E) at (6,0) {} ;
    \node[shape=circle,draw=black,fill=black, label=$e_6-e_7$] (F) at (7.5,0) {} ;
    \node[shape=circle,draw=black,fill=black, label=below:$e_7-e_8$] (G) at (9,0) {} ;
    \node[shape=circle,draw=black,fill=black, label=$\ell-123$] (H) at (3,1) {} ;
    \node[shape=circle,draw=black,fill=black, label=$e_8-e_9$] (I) at (10.5,0) {} ;

    \path [-,draw=black](A) edge node[left] {} (B);
    \path [-,draw=black](B) edge node[left] {} (C);
    \path [-,draw=black](C) edge node[left] {} (D);
    \path [-,draw=black](E) edge node[left] {} (F);
    \path [-,draw=black](F) edge node[left] {} (G);
    \path [-,draw=black](D) edge node[left] {} (E);
    \path [-,draw=black](C) edge node[left] {} (H);
    \path [-,draw=black](I) edge node[left] {} (G);
\end{tikzpicture}
\end{center}
The only $(-1)$-curve of this diagram is $e_9$.

By setting $\texttt{A=ME8}$ and running the Macaulay2 code in the Appendix, we can see that there is-- up to isomorphism of the graph-- one way to blow down the surface to $\P^2$. The only way to blow down this diagram is depicted in the following diagram.

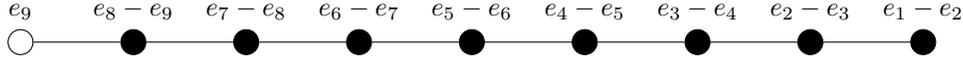
\begin{center}
\begin{tikzpicture}
    \node[shape=circle,draw=black,fill=none, label=$e_9$] (A) at (0,0) {} ;
    \node[shape=circle,draw=black,fill=black, label=$e_8-e_9$] (B) at (1.5,0) {} ;
    \node[shape=circle,draw=black,fill=black,label=$e_7-e_8$] (C) at (3,0) {} ;
    \node[shape=circle,draw=black,fill=black, label=$e_6-e_7$] (D) at (4.5,0) {} ;
    \node[shape=circle,draw=black,fill=black, label=$e_5-e_6$] (E) at (6,0) {} ;
    \node[shape=circle,draw=black,fill=black, label=$e_4-e_5$] (F) at (7.5,0) {} ;
    \node[shape=circle,draw=black,fill=black, label=$e_3-e_4$] (G) at (9,0) {} ;
    \node[shape=circle,draw=black,fill=black, label=$e_2-e_3$] (H) at (10.5,0) {} ;
    \node[shape=circle,draw=black,fill=black, label=$e_1-e_2$] (I) at (12,0) {} ;

    \path [-,draw=black](A) edge node[left] {} (B);
    \path [-,draw=black](D) edge node[left] {} (C);
    \path [-,draw=black](D) edge node[left] {} (E);
    \path [-,draw=black](F) edge node[left] {} (G);
    \path [-,draw=black](H) edge node[left] {} (I);
    \path [-,draw=black](E) edge node[left] {} (F);
    \path [-,draw=black](G) edge node[left] {} (H);
    \path [-,draw=black](B) edge node[left] {} (C);
\end{tikzpicture}
\end{center}
Since there is only one way to blow down this surface, there are no change of basis matrices.
\begin{itemize}
\item This diagram corresponds to a pencil spanned by a cuspidal cubic and a triple line that meets the cubic at its cusp with multiplicity 3. An example of such a pencil is that spanned by $x^3+y^2z$ and $y^3$.
\end{itemize}

\section{Unexpected Cubics in Characteristic 2}

Let $k$ be an algebraically closed field of characteristic 2, and let $\vhi\in k$ satisfy $\vhi^2+\vhi+1=0$.

We found in total eight types of configurations of points in $\P^2_k$ that yield unexpected cubics. Two come from the two blow-downs of $\wilde{A}_1^{\+8}$, three come from the four blow-downs of $\wilde{A}_1^{\+4}\+\wilde{D}_4$, and three come from the four blow-downs of $\wilde{A}_1^{\+2}\+\wilde{D}_6$. We itemize them below.

\begin{enumerate}

\item The two $\wilde{A}_1^{\+8}$ cases:
\begin{enumerate}

\item In the first $\wilde{A}_1^{\+8}$ case, the sheet of cubics spanned by $x^2(y+z)$, $y^2(x+z)$, and $z^2(x+y)$ produces cubic curves that contain the seven points $(1,0,0)\times 2$, $(0,1,0)\times 2$, $(0,0,1)\times 2$, and $(1,1,1)$, and have cusps at generic points.

\item In the second $\wilde{A}_1^{\+8}$ case (outlined in Akesseh's thesis), the sheet of cubics spanned by $xy(x+y)$, $xz(x+z)$, and $yz(y+z)$ produce cubic curves that contain the seven points of the Fano plane and have cusps at generic points.
\end{enumerate}

\item The three $\wilde{A}_1^{\+4}\+\wilde{D}_4$ cases:
\begin{enumerate}

\item In the first $\wilde{A}_1^{\+4}\+\wilde{D}_4$ blow-down, the sheet of cubics spanned by $x^2z+xy^2$, $x^2z+\vhi xz^2+\vhi^2 y^2z$, and $x^2z+xz^2+\vhi y^2z$ produce cubic curves containing the seven points $(0,0,1)\times 4$, $(1,0,0)\times 2$, and $(0,1,0)$, and have cusps at generic points.

\item In the second $\wilde{A}_1^{\+4}\+\wilde{D}_4$ blow-down, the sheet of cubics spanned by $(y^2+xz)(x+\vhi z)$, $xz(x+z)$, and $xz(x+\vhi^2 z)$ produces cubic curves that contain the seven points $(1,0,0)\times 2$, $(0,0,1)\times 2$, and $(0,1,0)\times 3$, and have cusps at generic points.

\item In the third $\wilde{A}_1^{\+4}\+\wilde{D}_4$ blow-down, the sheet of cubics spanned by $(y^2+xz)(x+\vhi z)$, $y^2(x+z)$, and $y^2(x+\vhi^2 z)$ produces cubic curves that contain the seven points $(0,0,1)\times 2$, $(\vhi,0,1)\times 2$, $(1,0,0)\times 2$, and $(0,1,0)$, and have cusps at generic points.
\end{enumerate}
\item The three $\wilde{A}_1^{\+2}\+\wilde{D}_6$ cases:
\begin{enumerate}

\item In the first $\wilde{A}_1^{\+2}\+\wilde{D}_6$ blow-down, the sheet of cubics spanned by $(y^2+xz)(x+z)$, $x^2z$, and $x^2(x+z)$ produces cubic curves that contain the seven points $(0,0,1)\times 4$ and $(0,1,0)\times 2$, and have cusps at generic points.

\item In the third $\wilde{A}_1^{\+2}\+\wilde{D}_6$ blow-down, the sheet of cubics spanned by $x^3+y^2z$, $xz(x+z)$, and $xz(x+\vhi z)$ produces a cubic curve that contains the seven points $(0,0,1)\times 2$, and $(0,1,0)\times 5$, and has a cusp at a generic point.

\item In the fourth $\wilde{A}_1^{\+2}\+\wilde{D}_6$ blow-down, the sheet of cubics spanned by $x(y^2+xz)$, $z(y^2+x^2+xz)$, and $(x+z)(y^2+\vhi x^2+xz)$ produces a cubic curve that contains the seven points $(0,0,1)\times 6$ and $(0,1,0)$, and has a cusp at a generic point.

\end{enumerate}
\end{enumerate}

The cuspidal curves produced by each sheet above is an unexpected cubic for a set $Z$ of seven points. The examples where some of the points of $Z$ are infinitely near are new.

\section{Quasi-Elliptic Fibrations in Characteristic 3}
Now let us consider quasi-elliptic fibrations in characteristic 3. According to Cossec and Dolgachev, a quasi-elliptic fibration in characteristic 3 will have a Dynkin diagram of $\wilde{A}_2^{\+4}$, $\wilde{A}_2\+\wilde{E}_6$, or $\wilde{E}_8$, \cite{CD}, Proposition 5.5.9, or \cite{CDL} Table 4.9.
\vspace{\baselineskip}
\begin{center}
    \textbf{Case 1:} $\wilde{A}_2^{\+4}$
\end{center}
Let us first consider the diagram $\wilde{A}_2^{\+4}$. We can label this diagram as follows.
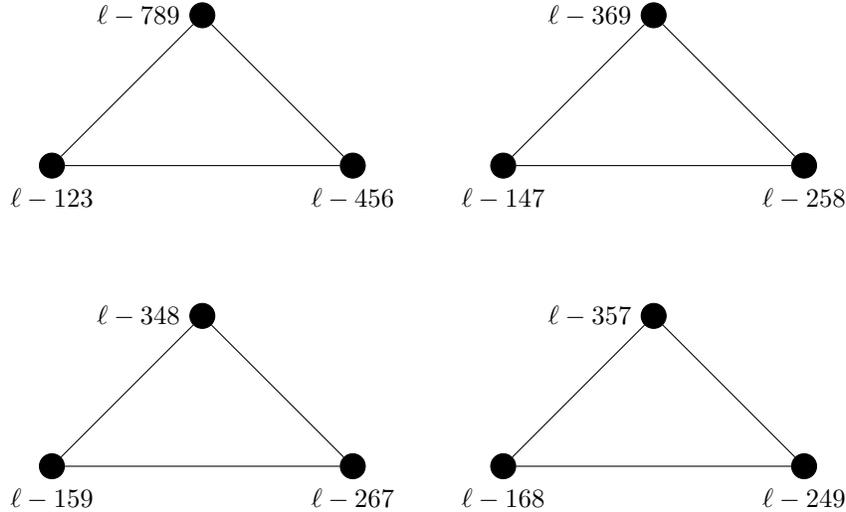
\begin{center}
\begin{tikzpicture}
    \node[shape=circle,draw=black,fill=black,label=below:$\ell-123$] (A) at (0,0) {};
    \node[shape=circle,draw=black,fill=black,label=below:$\ell-456$] (B) at (4,0) {};
    \node[shape=circle,draw=black,fill=black,label=left:$\ell-789$] (C) at (2,2) {};

    \node[shape=circle,draw=black,fill=black,label=below:$\ell-147$] (F) at (6,0) {};
    \node[shape=circle,draw=black,fill=black,label=below:$\ell-258$] (G) at (10,0) {};
    \node[shape=circle,draw=black,fill=black,label=left:$\ell-369$] (H) at (8,2) {};
    
     \node[shape=circle,draw=black,fill=black,label=below:$\ell-159$] (I) at (0,-4) {};
    \node[shape=circle,draw=black,fill=black,label=below:$\ell-267$] (J) at (4,-4) {};
    \node[shape=circle,draw=black,fill=black,label=left:$\ell-348$] (K) at (2,-2) {};

    \node[shape=circle,draw=black,fill=black,label=below:$\ell-168$] (L) at (6,-4) {};
    \node[shape=circle,draw=black,fill=black,label=below:$\ell-249$] (M) at (10,-4) {};
    \node[shape=circle,draw=black,fill=black,label=left:$\ell-357$] (N) at (8,-2) {};

    \path [-](B) edge node[left] {} (C);
    \path [-](A) edge node[left] {} (C);
    \path [-](A) edge node[left] {} (B);

    \path [-](F) edge node[left] {} (H);
    \path [-](G) edge node[left] {} (H);
    \path [-](G) edge node[left] {} (F);
    
    \path [-](I) edge node[left] {} (J);
    \path [-](J) edge node[left] {} (K);
    \path [-](K) edge node[left] {} (I);

    \path [-](L) edge node[left] {} (M);
    \path [-](M) edge node[left] {} (N);
    \path [-](N) edge node[left] {} (L);
\end{tikzpicture}
\end{center}
Using Lemma \ref{MW}, we know there are nine $(-1)$-curves; according to this labeling, they are $e_1,\dots,e_9$. 

We will show that there are-- up to isomorphism of the graph-- three ways to blow down the surface to $\P^2$. We have nine $(-1)$-curves which are all disjoint from each other and 12 $(-2)$-curves. To avoid ambiguity with $(-2)$-curves whose self-intersection may be raised to $-1$ during the blow-down process, we shall hereafter refer to these starting $(-1)$-curves as ``sections." First, observe that blowing down all nine sections yields a successful blow-down to the plane. Note that, in any of the succeeding scenarios, we cannot blow down more than two vertices on each $\wilde{A}_2$ component; that is because all $(-2)$-curves on each $\wilde{A}_2$ are adjacent to each other, so blowing down two will increase the self-intersection of the third to 0, so it cannot be blown down.

Suppose we were to blow down five of the sections. Then we cannot attain the remaining four blow-downs from the $(-2)$-curves; first note that we may not blow down more than one vertex for any of the $\wilde{A}_2$ components. This is because, if we are able to blow down a vertex on a $\wilde{A}_2$ component, then that vertex must intersect exactly one of the five blown-down sections, meaning that the other two vertices on the same $\wilde{A}_2$ must also intersect at least one, so now that component would not have any more $(-1)$-curves to blow down. Thus in order to blow down nine curves, we must choose exactly one vertex on each $\wilde{A}_2$ component. This, however, is not possible; by the structure of the graph, there must be at least one pair of $\wilde{A}_2$ components whose only $(-1)$-curves both intersect one of the five blown-down sections. We cannot blow down both of these vertices, because blowing down one will cause  the self-intersection of the other to increase. Therefore we cannot blow down four vertices from the $\wilde{A}_2$ components, so we must blow down a sixth section. This problem persists in the case of blowing down six, seven, and eight sections, so upon blowing down five sections, we are compelled to blow down all nine.

Now consider blowing down four sections. The only way to complete the blow-down to the plane by using five vertices on the $\wilde{A}_2$ components is by blowing down either two vertices on two components and one vertex on one, or two vertices on one component and one vertex on three. The former case is not possible; in order to blow down two vertices on one of the $\wilde{A}_2$ components, we must choose one vertex that meets one of the sections, followed by one vertex that meets none. That means that the three other sections all meet this $\wilde{A}_2$ component at the same vertex, and so they meet each of the other $\wilde{A}_2$ components at three different vertices, and thus we may not blow down two vertices on any of the other $\wilde{A}_2$ components. Thus the only way to blow down to the plane using only four sections is the latter path, which is achievable.

Now consider blowing down three sections. The only way to complete the blow-down to the plane by using six vertices on the $\wilde{A}_2$ components is by blowing down either two vertices on two components and one vertex on two, or two vertices on three component. The former case is not possible, because we cannot blow down vertices across four different $\wilde{A}_2$ components starting with only three sections; one of the $\wilde{A}_2$ will have no $(-1)$-curves to blow down. Thus the only way to blow down to the plane using only three sections is the latter path, which is achievable.

Finally, consider blowing down just two sections. Then we would need to blow down seven vertices on the $\wilde{A}_2$ components, which is impossible because we would need to blow down vertices distributed across at least three $\wilde{A}_2$ components. This cannot be done from blowing down just two sections, and so we would need to increase the number of sections we blow down, putting us in a previous case. The same is true from blowing down just one section, and so the three accounted blow-downs are a complete catalogue.

The three blow-downs of the diagram (up to isomorphism) are represented by the diagrams below.
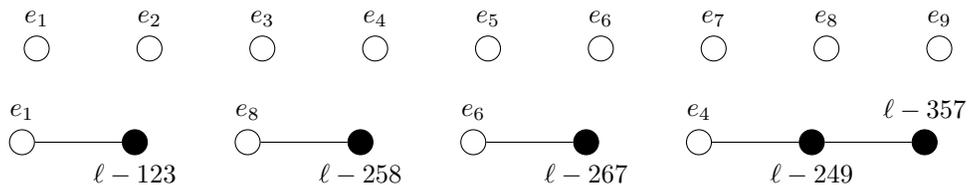
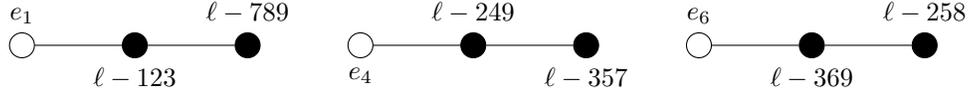
\begin{center}
\begin{tikzpicture}
    \node[shape=circle,draw=black,fill=none, label=$e_1$] (A) at (0,0) {} ;
    \node[shape=circle,draw=black,fill=none, label=$e_2$] (B) at (1.5,0) {} ;
    \node[shape=circle,draw=black,fill=none, label=$e_3$] (C) at (3,0) {} ;
    \node[shape=circle,draw=black,fill=none, label=$e_4$] (D) at (4.5,0) {} ;
    \node[shape=circle,draw=black,fill=none, label=$e_5$] (E) at (6,0) {} ;
    \node[shape=circle,draw=black,fill=none, label=$e_6$] (F) at (7.5,0) {} ;
    \node[shape=circle,draw=black,fill=none, label=$e_7$] (G) at (9,0) {} ;
    \node[shape=circle,draw=black,fill=none, label=$e_8$] (H) at (10.5,0) {} ;
    \node[shape=circle,draw=black,fill=none, label=$e_9$] (I) at (12,0) {} ;

\end{tikzpicture}
\end{center}
\begin{center}
\begin{tikzpicture}
    \node[shape=circle,draw=black,fill=none, label=$e_1$] (A) at (0,0) {} ;
    \node[shape=circle,draw=black,fill=black, label=below:$\ell-123$] (B) at (1.5,0) {} ;
    \node[shape=circle,draw=black,fill=none,label=$e_8$] (C) at (3,0) {} ;
    \node[shape=circle,draw=black,fill=black, label=below:$\ell-258$] (D) at (4.5,0) {} ;
    \node[shape=circle,draw=black,fill=none, label=$e_6$] (E) at (6,0) {} ;
    \node[shape=circle,draw=black,fill=black, label=below:$\ell-267$] (F) at (7.5,0) {} ;
    \node[shape=circle,draw=black,fill=none, label=$e_4$] (G) at (9,0) {} ;
    \node[shape=circle,draw=black,fill=black, label=below:$\ell-249$] (H) at (10.5,0) {} ;
    \node[shape=circle,draw=black,fill=black, label=$\ell-357$] (I) at (12,0) {} ;

    \path [-,draw=black](A) edge node[left] {} (B);
    \path [-,draw=black](D) edge node[left] {} (C);
  
    \path [-,draw=black](H) edge node[left] {} (I);
    \path [-,draw=black](E) edge node[left] {} (F);
    \path [-,draw=black](G) edge node[left] {} (H);
\end{tikzpicture}
\end{center}
\begin{center}
\begin{tikzpicture}
    \node[shape=circle,draw=black,fill=none, label=$e_1$] (A) at (0,0) {} ;
    \node[shape=circle,draw=black,fill=black, label=below:$\ell-123$] (B) at (1.5,0) {} ;
    \node[shape=circle,draw=black,fill=black,label=$\ell-789$] (C) at (3,0) {} ;
    \node[shape=circle,draw=black,fill=none, label=below:$e_4$] (D) at (4.5,0) {} ;
    \node[shape=circle,draw=black,fill=black, label=$\ell-249$] (E) at (6,0) {} ;
    \node[shape=circle,draw=black,fill=black, label=below:$\ell-357$] (F) at (7.5,0) {} ;
    \node[shape=circle,draw=black,fill=none, label=$e_6$] (G) at (9,0) {} ;
    \node[shape=circle,draw=black,fill=black, label=below:$\ell-369$] (H) at (10.5,0) {} ;
    \node[shape=circle,draw=black,fill=black, label=$\ell-258$] (I) at (12,0) {} ;

    \path [-,draw=black](A) edge node[left] {} (B);
    \path [-,draw=black](B) edge node[left] {} (C);
    \path [-,draw=black](D) edge node[left] {} (E);
    \path [-,draw=black](H) edge node[left] {} (I);
    \path [-,draw=black](E) edge node[left] {} (F);
    \path [-,draw=black](G) edge node[left] {} (H);
\end{tikzpicture}
\end{center}
The second two blow-downs give us the change-of-basis matrices $$\left(\begin{matrix}
3 & 0 & 2 & 1 & 0 & 1 & 0 & 1 & 0 & 1 \\
-1 & 0 & -1 & -1 & 0 & 0 & 0 & 0 & 0 & 0 \\
0 & 1 & 0 & 0 & 0 & 0 & 0 & 0 & 0 & 0 \\
-1 & 0 & -1 & 0 & 0 & -1 & 0 & 0 & 0 & 0 \\
0 & 0 & 0 & 0 & 0 & 0 & 0 & 0 & 1 & 0 \\
-1 & 0 & -1 & 0 & 0 & 0 & 0 & -1 & 0 & 0 \\
0 & 0 & 0 & 0 & 0 & 0 & 1 & 0 & 0 & 0 \\
-2 & 0 & -1 & -1 & 0 & -1 & 0 & -1 & 0 & -1 \\
-1 & 0 & -1 & 0 & 0 & 0 & 0 & 0 & 0 & -1 \\
0 & 0 & 0 & 0 & 1 & 0 & 0 & 0 & 0 & 0
\end{matrix}\right):\Pic X\to\Pic X$$ and $$\left(\begin{matrix}
4 & 0 & 2 & 2 & 0 & 1 & 0 & 1 & 1 & 2 \\
-2 & 0 & -1 & -1 & 0 & 0 & 0 & -1 & -1 & -1 \\
-1 & 0 & -1 & -1 & 0 & 0 & 0 & 0 & 0 & 0 \\
0 & 1 & 0 & 0 & 0 & 0 & 0 & 0 & 0 & 0 \\
-2 & 0 & -1 & -1 & 0 & -1 & 0 & -1 & 0 & -1 \\
-1 & 0 & -1 & 0 & 0 & 0 & 0 & 0 & 0 & -1 \\
0 & 0 & 0 & 0 & 1 & 0 & 0 & 0 & 0 & 0 \\
-2 & 0 & -1 & -1 & 0 & -1 & 0 & 0 & -1 & -1 \\
-1 & 0 & 0 & -1 & 0 & 0 & 0 & 0 & 0 & -1 \\
0 & 0 & 0 & 0 & 0 & 0 & 1 & 0 & 0 & 0
\end{matrix}\right):\Pic X\to\Pic X$$
respectively. Applying these matrices to the Dynkin diagram above gives us the following relabelings.
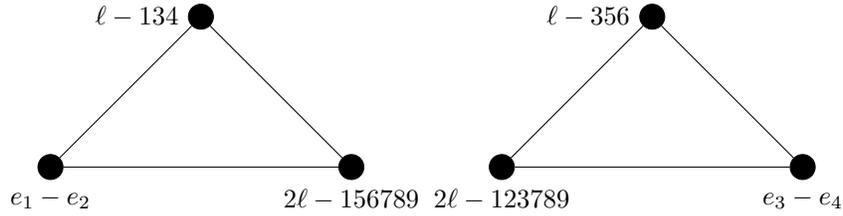
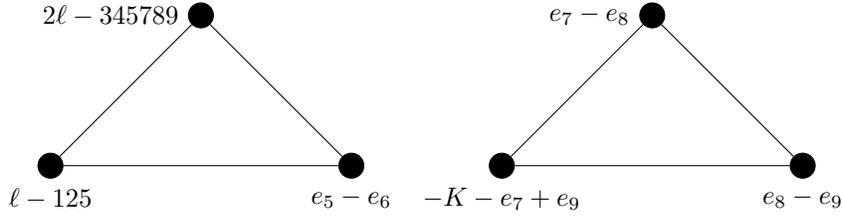
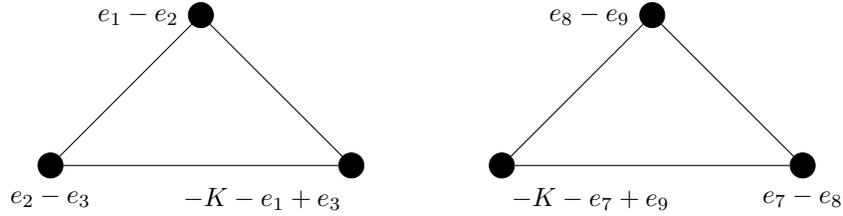
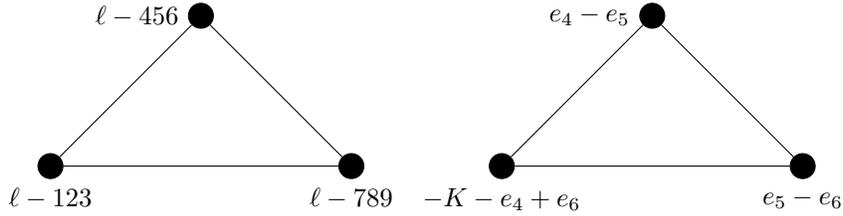
\begin{center}
\begin{tikzpicture}
    \node[shape=circle,draw=black,fill=black,label=below:$e_1-e_2$] (A) at (0,0) {};
    \node[shape=circle,draw=black,fill=black,label=below:$2\ell-156789$] (B) at (4,0) {};
    \node[shape=circle,draw=black,fill=black,label=left:$\ell-134$] (C) at (2,2) {};

    \node[shape=circle,draw=black,fill=black,label=below:$2\ell-123789$] (F) at (6,0) {};
    \node[shape=circle,draw=black,fill=black,label=below:$e_3-e_4$] (G) at (10,0) {};
    \node[shape=circle,draw=black,fill=black,label=left:$\ell-356$] (H) at (8,2) {};
    
     \node[shape=circle,draw=black,fill=black,label=below:$\ell-125$] (I) at (0,-4) {};
    \node[shape=circle,draw=black,fill=black,label=below:$e_5-e_6$] (J) at (4,-4) {};
    \node[shape=circle,draw=black,fill=black,label=left:$2\ell-345789$] (K) at (2,-2) {};

    \node[shape=circle,draw=black,fill=black,label=below:$-K-e_7+e_9$] (L) at (6,-4) {};
    \node[shape=circle,draw=black,fill=black,label=below:$e_8-e_9$] (M) at (10,-4) {};
    \node[shape=circle,draw=black,fill=black,label=left:$e_7-e_8$] (N) at (8,-2) {};

    \path [-](B) edge node[left] {} (C);
    \path [-](A) edge node[left] {} (C);
    \path [-](A) edge node[left] {} (B);

    \path [-](F) edge node[left] {} (H);
    \path [-](G) edge node[left] {} (H);
    \path [-](G) edge node[left] {} (F);
    
    \path [-](I) edge node[left] {} (J);
    \path [-](J) edge node[left] {} (K);
    \path [-](K) edge node[left] {} (I);

    \path [-](L) edge node[left] {} (M);
    \path [-](M) edge node[left] {} (N);
    \path [-](N) edge node[left] {} (L);
\end{tikzpicture}
\end{center}
\begin{center}
\begin{tikzpicture}
    \node[shape=circle,draw=black,fill=black,label=below:$e_2-e_3$] (A) at (0,0) {};
    \node[shape=circle,draw=black,fill=black,label=267:$-K-e_1+e_3$] (B) at (4,0) {};
    \node[shape=circle,draw=black,fill=black,label=left:$e_1-e_2$] (C) at (2,2) {};

    \node[shape=circle,draw=black,fill=black,label=273:$-K-e_7+e_9$] (F) at (6,0) {};
    \node[shape=circle,draw=black,fill=black,label=below:$e_7-e_8$] (G) at (10,0) {};
    \node[shape=circle,draw=black,fill=black,label=left:$e_8-e_9$] (H) at (8,2) {};
    
     \node[shape=circle,draw=black,fill=black,label=below:$\ell-123$] (I) at (0,-4) {};
    \node[shape=circle,draw=black,fill=black,label=below:$\ell-789$] (J) at (4,-4) {};
    \node[shape=circle,draw=black,fill=black,label=left:$\ell-456$] (K) at (2,-2) {};

    \node[shape=circle,draw=black,fill=black,label=below:$-K-e_4+e_6$] (L) at (6,-4) {};
    \node[shape=circle,draw=black,fill=black,label=below:$e_5-e_6$] (M) at (10,-4) {};
    \node[shape=circle,draw=black,fill=black,label=left:$e_4-e_5$] (N) at (8,-2) {};

    \path [-](B) edge node[left] {} (C);
    \path [-](A) edge node[left] {} (C);
    \path [-](A) edge node[left] {} (B);

    \path [-](F) edge node[left] {} (H);
    \path [-](G) edge node[left] {} (H);
    \path [-](G) edge node[left] {} (F);
    
    \path [-](I) edge node[left] {} (J);
    \path [-](J) edge node[left] {} (K);
    \path [-](K) edge node[left] {} (I);

    \path [-](L) edge node[left] {} (M);
    \path [-](M) edge node[left] {} (N);
    \path [-](N) edge node[left] {} (L);
\end{tikzpicture}
\end{center}
\begin{itemize}
\item The original Dynkin diagram corresponds to a pencil spanned by the union of three concurrent lines and the union of three different concurrent lines. An example of such a pencil is that spanned by $x(x+z)(x-z)$ and $y(y+z)(y-z)$.

\item The second diagram corresponds to a pencil spanned by the union of a conic and a line, and the union of a line tangent to the first conic and a conic that contains the intersection of the two lines and the intersection of the first conic and the first line. An example of such a pencil is that spanned by $2x^2z+yz^2$ and $xy^2+2yz^2$.

\item The third diagram corresponds to a pencil spanned by a flexible cuspidal cubic and the union of three lines, two of which are tangent to flex points on the cubic and one which intersects the cusp with multiplicity 3. An example of such a pencil is that spanned by $x^3-y^2z$ and $yz(y-z)$.
\end{itemize}
\begin{center}
    \textbf{Case 2:}  $\wilde{A}_2\+\wilde{E}_6$
\end{center}
Let us now consider the diagram $\wilde{A}_2\+\wilde{E}_6$. A labeling is pictured below.
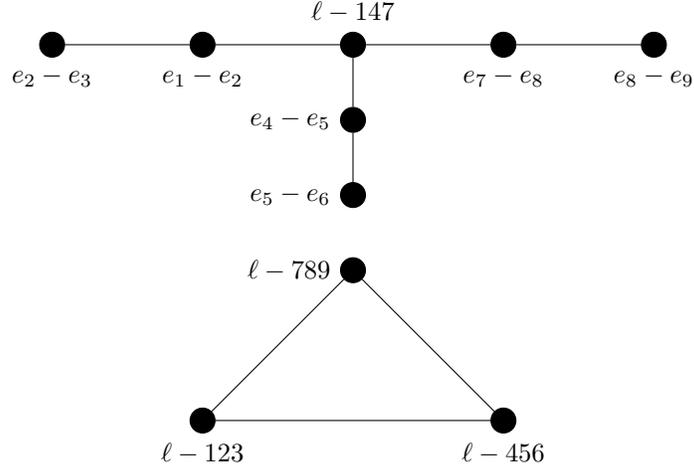
\begin{center}
\begin{tikzpicture}
    \node[shape=circle,draw=black,fill=black,label=below:$\ell-123$] (A) at (0,0) {};
    \node[shape=circle,draw=black,fill=black,label=below:$\ell-456$] (B) at (4,0) {};
    \node[shape=circle,draw=black,fill=black,label=left:$\ell-789$] (C) at (2,2) {};
    
    \node[shape=circle,draw=black,fill=black,label=below:$e_2-e_3$] (D) at (-2,5) {};
    \node[shape=circle,draw=black,fill=black,label=below:$e_1-e_2$] (E) at (0,5) {};
    \node[shape=circle,draw=black,fill=black,label=$\ell-147$] (F) at (2,5) {};
    \node[shape=circle,draw=black,fill=black,label=below:$e_7-e_8$] (G) at (4,5) {};
    \node[shape=circle,draw=black,fill=black,label=below:$e_8-e_9$] (H) at (6,5) {};
    \node[shape=circle,draw=black,fill=black,label=left:$e_4-e_5$] (I) at (2,4) {};
    \node[shape=circle,draw=black,fill=black,label=left:$e_5-e_6$] (J) at (2,3) {};

    \path [-](B) edge node[left] {} (C);
    \path [-](A) edge node[left] {} (C);
    \path [-](A) edge node[left] {} (B);
    
    \path [-](D) edge node[left] {} (E);
    \path [-](E) edge node[left] {} (F);
    \path [-](F) edge node[left] {} (G);
     \path [-](G) edge node[left] {} (H);
    \path [-](F) edge node[left] {} (I);
    \path [-](I) edge node[left] {} (J);

\end{tikzpicture}
\end{center}
Using Lemma \ref{MW}, we know this surface has three $(-1)$-curves; they are $e_3$, $e_6$, and $e_9$. 

By setting $\texttt{A=MA2E6}$ and running the Macaulay2 code in the Appendix, we can see that there are-- up to isomorphism of the graph-- three ways to blow down the surface to $\P^2$. They are depicted below.
\begin{center}
\begin{tikzpicture}
    \node[shape=circle,draw=black,fill=none, label=$e_3$] (A) at (0,0) {} ;
    \node[shape=circle,draw=black,fill=black, label=below:$e_2-e_3$] (B) at (1.5,0) {} ;
    \node[shape=circle,draw=black,fill=black,label=$e_1-e_2$] (C) at (3,0) {} ;
    \node[shape=circle,draw=black,fill=none, label=below:$e_6$] (D) at (4.5,0) {} ;
    \node[shape=circle,draw=black,fill=black, label=$e_5-e_6$] (E) at (6,0) {} ;
    \node[shape=circle,draw=black,fill=black, label=below:$e_4-e_5$] (F) at (7.5,0) {} ;
    \node[shape=circle,draw=black,fill=none, label=$e_9$] (G) at (9,0) {} ;
    \node[shape=circle,draw=black,fill=black, label=below:$e_8-e_9$] (H) at (10.5,0) {} ;
    \node[shape=circle,draw=black,fill=black, label=$e_7-e_8$] (I) at (12,0) {} ;

    \path [-,draw=black](A) edge node[left] {} (B);
    \path [-,draw=black](B) edge node[left] {} (C);
    \path [-,draw=black](D) edge node[left] {} (E);
    \path [-,draw=black](H) edge node[left] {} (I);
    \path [-,draw=black](E) edge node[left] {} (F);
    \path [-,draw=black](G) edge node[left] {} (H);
\end{tikzpicture}
\end{center}
\begin{center}
\begin{tikzpicture}
    \node[shape=circle,draw=black,fill=none, label=$e_6$] (A) at (0,0) {} ;
    \node[shape=circle,draw=black,fill=black, label=$e_5-e_6$] (B) at (1.5,0) {} ;
    \node[shape=circle,draw=black,fill=black, label=$e_4-e_5$] (C) at (3,0) {} ;
    \node[shape=circle,draw=black,fill=black, label=below:$\ell-147$] (D) at (4.5,0) {} ;
    \node[shape=circle,draw=black,fill=black, label=$e_1-e_2$] (E) at (6,0) {} ;
    \node[shape=circle,draw=black,fill=black, label=$e_2-e_3$] (F) at (7.5,0) {} ;
    \node[shape=circle,draw=black,fill=none, label=$e_9$] (G) at (9,0) {} ;
    \node[shape=circle,draw=black,fill=black, label=below:$\ell-789$] (H) at (10.5,0) {} ;
    \node[shape=circle,draw=black,fill=black, label=$\ell-123$] (I) at (12,0) {} ;
    
     \path [-,draw=black](A) edge node[left] {} (B);
    \path [-,draw=black](B) edge node[left] {} (C);
    \path [-,draw=black](C) edge node[left] {} (D);
    \path [-,draw=black](D) edge node[left] {} (E);
    \path [-,draw=black](E) edge node[left] {} (F);
    \path [-,draw=black](G) edge node[left] {} (H);
    \path [-,draw=black](I) edge node[left] {} (H);

\end{tikzpicture}
\end{center}
\begin{center}
\begin{tikzpicture}
    \node[shape=circle,draw=black,fill=none, label=$e_3$] (A) at (0,0) {} ;
    \node[shape=circle,draw=black,fill=black, label=below:$e_2-e_3$] (B) at (1.5,0) {} ;
    \node[shape=circle,draw=black,fill=black,label=$e_1-e_2$] (C) at (3,0) {} ;
    \node[shape=circle,draw=black,fill=black, label=below:$\ell-147$] (D) at (4.5,0) {} ;
    \node[shape=circle,draw=black,fill=black, label=$e_4-e_5$] (E) at (6,0) {} ;
    \node[shape=circle,draw=black,fill=none, label=below:$e_6$] (F) at (7.5,0) {} ;
    \node[shape=circle,draw=black,fill=black, label=$\ell-456$] (G) at (9,0) {} ;
    \node[shape=circle,draw=black,fill=none, label=below:$e_9$] (H) at (10.5,0) {} ;
    \node[shape=circle,draw=black,fill=black, label=$e_8-e_9$] (I) at (12,0) {} ;

    \path [-,draw=black](A) edge node[left] {} (B);
    \path [-,draw=black](D) edge node[left] {} (C);
  \path [-,draw=black](C) edge node[left] {} (B);
  \path [-,draw=black](E) edge node[left] {} (D);
  \path [-,draw=black](F) edge node[left] {} (G);
  \path [-,draw=black](H) edge node[left] {} (I);
\end{tikzpicture}
\end{center}
The second two blow-downs give us the change-of-basis matrices $$\left(\begin{matrix}
3 & 1 & 1 & 1 & 0 & 0 & 0 & 2 & 1 & 0 \\
-1 & 0 & 0 & -1 & 0 & 0 & 0 & -1 & 0 & 0 \\
-1 & 0 & -1 & 0 & 0 & 0 & 0 & -1 & 0 & 0 \\
-1 & -1 & 0 & 0 & 0 & 0 & 0 & -1 & 0 & 0 \\
0 & 0 & 0 & 0 & 1 & 0 & 0 & 0 & 0 & 0 \\
0 & 0 & 0 & 0 & 0 & 1 & 0 & 0 & 0 & 0 \\
0 & 0 & 0 & 0 & 0 & 0 & 1 & 0 & 0 & 0 \\
-2 & -1 & -1 & -1 & 0 & 0 & 0 & -1 & -1 & 0 \\
-1 & 0 & 0 & 0 & 0 & 0 & 0 & -1 & -1 & 0 \\
0 & 0 & 0 & 0 & 0 & 0 & 0 & 0 & 0 & 1
\end{matrix}\right):\Pic X\to\Pic X$$ and $$\left(\begin{matrix}
2 & 0 & 0 & 0 & 1 & 1 & 0 & 1 & 0 & 0 \\
-1 & 0 & 0 & 0 & 0 & -1 & 0 & -1 & 0 & 0 \\
-1 & 0 & 0 & 0 & -1 & 0 & 0 & -1 & 0 & 0 \\
0 & 1 & 0 & 0 & 0 & 0 & 0 & 0 & 0 & 0 \\
0 & 0 & 1 & 0 & 0 & 0 & 0 & 0 & 0 & 0 \\
0 & 0 & 0 & 1 & 0 & 0 & 0 & 0 & 0 & 0 \\
-1 & 0 & 0 & 0 & -1 & -1 & 0 & 0 & 0 & 0 \\
0 & 0 & 0 & 0 & 0 & 0 & 1 & 0 & 0 & 0 \\
0 & 0 & 0 & 0 & 0 & 0 & 0 & 0 & 1 & 0 \\
0 & 0 & 0 & 0 & 0 & 0 & 0 & 0 & 0 & 1
\end{matrix}\right):\Pic X\to\Pic X$$ respectively. Applying these matrices to the Dynkin diagram gives us the relabelings depicted below.
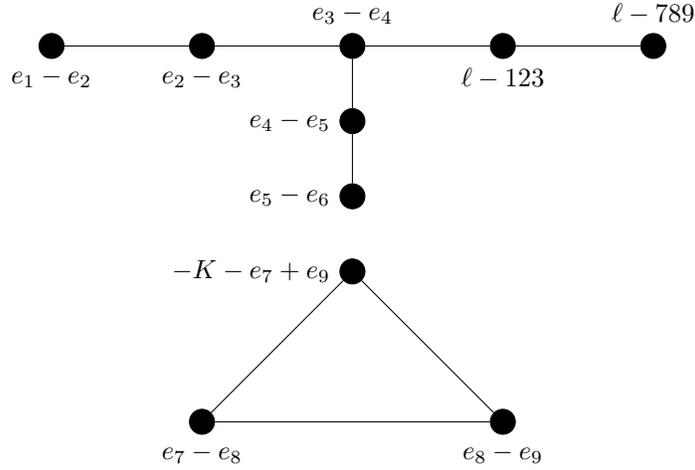
\begin{center}
\begin{tikzpicture}
    \node[shape=circle,draw=black,fill=black,label=below:$e_7-e_8$] (A) at (0,0) {};
    \node[shape=circle,draw=black,fill=black,label=below:$e_8-e_9$] (B) at (4,0) {};
    \node[shape=circle,draw=black,fill=black,label=left:$-K-e_7+e_9$] (C) at (2,2) {};
    
    \node[shape=circle,draw=black,fill=black,label=below:$e_1-e_2$] (D) at (-2,5) {};
    \node[shape=circle,draw=black,fill=black,label=below:$e_2-e_3$] (E) at (0,5) {};
    \node[shape=circle,draw=black,fill=black,label=$e_3-e_4$] (F) at (2,5) {};
    \node[shape=circle,draw=black,fill=black,label=below:$\ell-123$] (G) at (4,5) {};
    \node[shape=circle,draw=black,fill=black,label=$\ell-789$] (H) at (6,5) {};
    \node[shape=circle,draw=black,fill=black,label=left:$e_4-e_5$] (I) at (2,4) {};
    \node[shape=circle,draw=black,fill=black,label=left:$e_5-e_6$] (J) at (2,3) {};

    \path [-](B) edge node[left] {} (C);
    \path [-](A) edge node[left] {} (C);
    \path [-](A) edge node[left] {} (B);
    
    \path [-](D) edge node[left] {} (E);
    \path [-](E) edge node[left] {} (F);
    \path [-](F) edge node[left] {} (G);
     \path [-](G) edge node[left] {} (H);
    \path [-](F) edge node[left] {} (I);
    \path [-](I) edge node[left] {} (J);

\end{tikzpicture}
\end{center}
\noindent\makebox[\linewidth]{\rule{\textwidth}{0.4pt}}
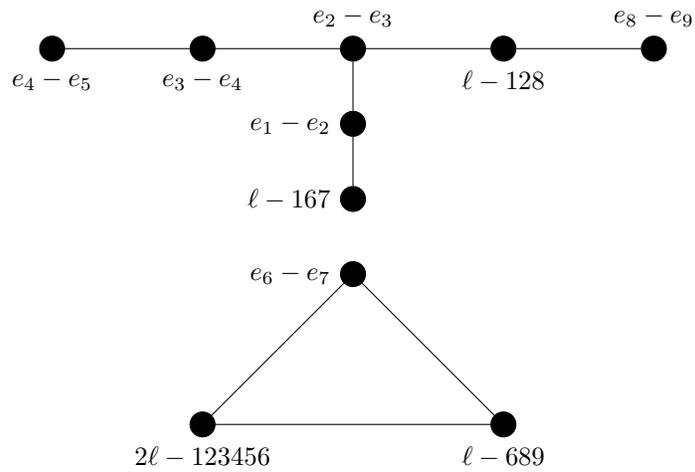
\begin{center}
\begin{tikzpicture}
    \node[shape=circle,draw=black,fill=black,label=below:$2\ell-123456$] (A) at (0,0) {};
    \node[shape=circle,draw=black,fill=black,label=below:$\ell-689$] (B) at (4,0) {};
    \node[shape=circle,draw=black,fill=black,label=left:$e_6-e_7$] (C) at (2,2) {};
    
    \node[shape=circle,draw=black,fill=black,label=below:$e_4-e_5$] (D) at (-2,5) {};
    \node[shape=circle,draw=black,fill=black,label=below:$e_3-e_4$] (E) at (0,5) {};
    \node[shape=circle,draw=black,fill=black,label=$e_2-e_3$] (F) at (2,5) {};
    \node[shape=circle,draw=black,fill=black,label=below:$\ell-128$] (G) at (4,5) {};
    \node[shape=circle,draw=black,fill=black,label=$e_8-e_9$] (H) at (6,5) {};
    \node[shape=circle,draw=black,fill=black,label=left:$e_1-e_2$] (I) at (2,4) {};
    \node[shape=circle,draw=black,fill=black,label=left:$\ell-167$] (J) at (2,3) {};

    \path [-](B) edge node[left] {} (C);
    \path [-](A) edge node[left] {} (C);
    \path [-](A) edge node[left] {} (B);
    
    \path [-](D) edge node[left] {} (E);
    \path [-](E) edge node[left] {} (F);
    \path [-](F) edge node[left] {} (G);
     \path [-](G) edge node[left] {} (H);
    \path [-](F) edge node[left] {} (I);
    \path [-](I) edge node[left] {} (J);

\end{tikzpicture}
\end{center}
\begin{itemize}
\item The original diagram corresponds to a pencil spanned by the union of three lines and a triple line. An example of such a pencil is that spanned by $xyz$ and $(x+y+z)^3$.

\item The second diagram corresponds to a pencil spanned by a flexible cuspidal cubic and the union of a line and a double line, with the line meeting the cubic at a flex point with multiplicity three, and the double line intersecting the cusp of the cubic with multiplicity three.

\item The third diagram corresponds to a pencil spanned by the union of a conic and a line and the union of a line and a double line, with the line containing one of the intersection points of the conic and the first line, and the double line is tangent to the conic at the other point it meets the second line. An example of such a pencil is that spanned by $x^2(x-y)$ and $(2xy+2xz+y^2)z$.
\end{itemize}
\begin{center}
    \textbf{Case 3:} $\wilde{E}_8$
\end{center}
Let us now consider the Dynkin diagram $\wilde{E}_8$, as depicted below.
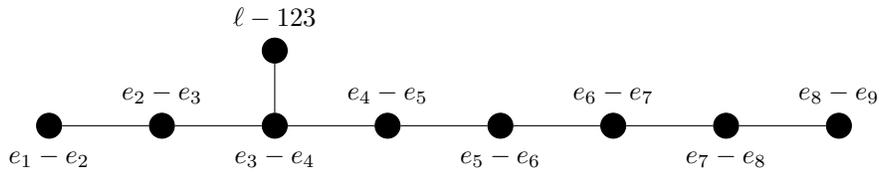
\begin{center}
\begin{tikzpicture}
    \node[shape=circle,draw=black,fill=black, label=below:$e_1-e_2$] (A) at (0,0) {} ;
    \node[shape=circle,draw=black,fill=black, label=$e_2-e_3$] (B) at (1.5,0) {} ;
    \node[shape=circle,draw=black,fill=black, label=below:$e_3-e_4$] (C) at (3,0) {} ;
    \node[shape=circle,draw=black,fill=black, label=$e_4-e_5$] (D) at (4.5,0) {} ;
    \node[shape=circle,draw=black,fill=black, label=below:$e_5-e_6$] (E) at (6,0) {} ;
    \node[shape=circle,draw=black,fill=black, label=$e_6-e_7$] (F) at (7.5,0) {} ;
    \node[shape=circle,draw=black,fill=black, label=below:$e_7-e_8$] (G) at (9,0) {} ;
    \node[shape=circle,draw=black,fill=black, label=$\ell-123$] (H) at (3,1) {} ;
    \node[shape=circle,draw=black,fill=black, label=$e_8-e_9$] (I) at (10.5,0) {} ;

    \path [-,draw=black](A) edge node[left] {} (B);
    \path [-,draw=black](B) edge node[left] {} (C);
    \path [-,draw=black](C) edge node[left] {} (D);
    \path [-,draw=black](E) edge node[left] {} (F);
    \path [-,draw=black](F) edge node[left] {} (G);
    \path [-,draw=black](D) edge node[left] {} (E);
    \path [-,draw=black](C) edge node[left] {} (H);
    \path [-,draw=black](I) edge node[left] {} (G);
\end{tikzpicture}
\end{center}
Using Lemma \ref{MW}, we know there is only one $(-1)$-curve of this diagram; it is $e_9$. 

By setting $\texttt{A=ME8}$ and running the Macaulay2 code in the Appendix, we can see that there is-- up to isomorphism of the graph-- one way to blow down the surface to $\P^2$. The only way to blow down this diagram is depicted in the following diagram.
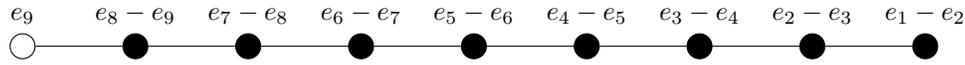
\begin{center}
\begin{tikzpicture}
    \node[shape=circle,draw=black,fill=none, label=$e_9$] (A) at (0,0) {} ;
    \node[shape=circle,draw=black,fill=black, label=$e_8-e_9$] (B) at (1.5,0) {} ;
    \node[shape=circle,draw=black,fill=black,label=$e_7-e_8$] (C) at (3,0) {} ;
    \node[shape=circle,draw=black,fill=black, label=$e_6-e_7$] (D) at (4.5,0) {} ;
    \node[shape=circle,draw=black,fill=black, label=$e_5-e_6$] (E) at (6,0) {} ;
    \node[shape=circle,draw=black,fill=black, label=$e_4-e_5$] (F) at (7.5,0) {} ;
    \node[shape=circle,draw=black,fill=black, label=$e_3-e_4$] (G) at (9,0) {} ;
    \node[shape=circle,draw=black,fill=black, label=$e_2-e_3$] (H) at (10.5,0) {} ;
    \node[shape=circle,draw=black,fill=black, label=$e_1-e_2$] (I) at (12,0) {} ;

    \path [-,draw=black](A) edge node[left] {} (B);
    \path [-,draw=black](D) edge node[left] {} (C);
    \path [-,draw=black](D) edge node[left] {} (E);
    \path [-,draw=black](F) edge node[left] {} (G);
    \path [-,draw=black](H) edge node[left] {} (I);
    \path [-,draw=black](E) edge node[left] {} (F);
    \path [-,draw=black](G) edge node[left] {} (H);
    \path [-,draw=black](B) edge node[left] {} (C);
\end{tikzpicture}
\end{center}
Since there is only one way to blow down this surface, there are no change of basis matrices.
\begin{itemize}
\item This diagram corresponds to a pencil spanned by a cuspidal cubic and a triple line that meets the cubic at its cusp with multiplicity 3. An example of such a pencil is that spanned by $x^3+y^2z$ and $y^3$.
\end{itemize}
\begin{thm}\label{thm3}
Let $k$ be a field of characteristic 3. There are no unexpected cubic curves in $\P^2_k$ coming from the quasi-elliptic fibrations listed above.
\end{thm}
\begin{proof}
Recall that an unexpected cubic curve comes from choosing seven points $P_1,P_2,\dots,P_7\in\P^2$ such that there is a cubic with a double point at the general point $P$. By Proposition \ref{uq}, an unexpected cubic yields a quasi-elliptic fibration that is isomorphic to the blowup of $\P^2_k$ at the following nine points: the seven points $P_1,\dots,P_7$, a general point $P$ and a point infinitely-near $P$. Since $P$ is general, it will not be equal to any of the $P_i$. Therefore we need only consider the quasi-elliptic fibrations that are attained by blowing up $\P^2_k$ at some point exactly twice, with the other seven blowup points elsewhere. Those are the second blow-down of $\wilde{A}_2^{\oplus 4}$ and the third blow-down of $\wilde{A}_2\oplus\wilde{E}_6$. Notice that in precisely those two blow-down diagrams do we have any connected components comprising a single white vertex and a single black vertex, corresponding to a point blowing up twice. Thus we need only check that quasi-elliptic fibrations with the Dynkin diagrams
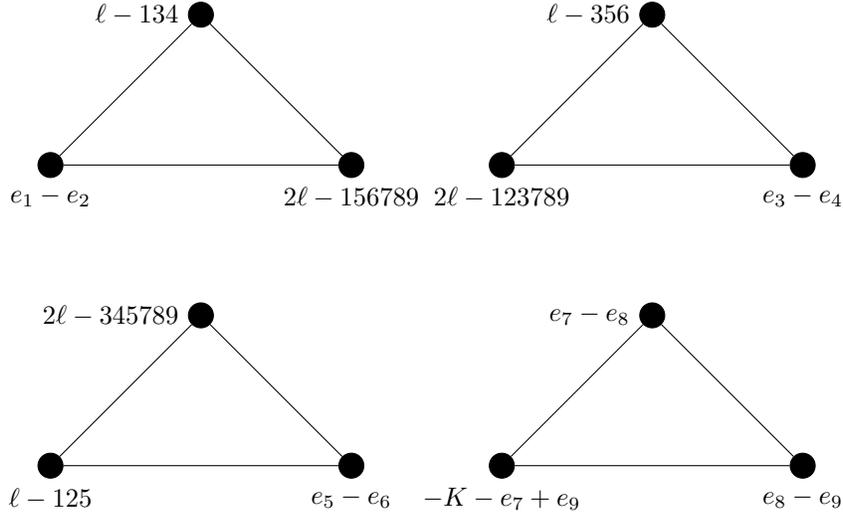
\begin{center}
\begin{tikzpicture}
    \node[shape=circle,draw=black,fill=black,label=below:$e_1-e_2$] (A) at (0,0) {};
    \node[shape=circle,draw=black,fill=black,label=below:$2\ell-156789$] (B) at (4,0) {};
    \node[shape=circle,draw=black,fill=black,label=left:$\ell-134$] (C) at (2,2) {};

    \node[shape=circle,draw=black,fill=black,label=below:$2\ell-123789$] (F) at (6,0) {};
    \node[shape=circle,draw=black,fill=black,label=below:$e_3-e_4$] (G) at (10,0) {};
    \node[shape=circle,draw=black,fill=black,label=left:$\ell-356$] (H) at (8,2) {};
    
     \node[shape=circle,draw=black,fill=black,label=below:$\ell-125$] (I) at (0,-4) {};
    \node[shape=circle,draw=black,fill=black,label=below:$e_5-e_6$] (J) at (4,-4) {};
    \node[shape=circle,draw=black,fill=black,label=left:$2\ell-345789$] (K) at (2,-2) {};

    \node[shape=circle,draw=black,fill=black,label=below:$-K-e_7+e_9$] (L) at (6,-4) {};
    \node[shape=circle,draw=black,fill=black,label=below:$e_8-e_9$] (M) at (10,-4) {};
    \node[shape=circle,draw=black,fill=black,label=left:$e_7-e_8$] (N) at (8,-2) {};

    \path [-](B) edge node[left] {} (C);
    \path [-](A) edge node[left] {} (C);
    \path [-](A) edge node[left] {} (B);

    \path [-](F) edge node[left] {} (H);
    \path [-](G) edge node[left] {} (H);
    \path [-](G) edge node[left] {} (F);
    
    \path [-](I) edge node[left] {} (J);
    \path [-](J) edge node[left] {} (K);
    \path [-](K) edge node[left] {} (I);

    \path [-](L) edge node[left] {} (M);
    \path [-](M) edge node[left] {} (N);
    \path [-](N) edge node[left] {} (L);
\end{tikzpicture}
\end{center}
and 
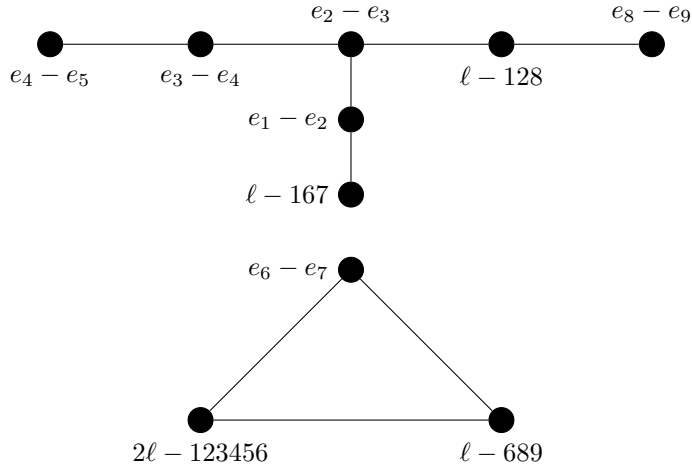
\begin{center}
\begin{tikzpicture}
    \node[shape=circle,draw=black,fill=black,label=below:$2\ell-123456$] (A) at (0,0) {};
    \node[shape=circle,draw=black,fill=black,label=below:$\ell-689$] (B) at (4,0) {};
    \node[shape=circle,draw=black,fill=black,label=left:$e_6-e_7$] (C) at (2,2) {};
    
    \node[shape=circle,draw=black,fill=black,label=below:$e_4-e_5$] (D) at (-2,5) {};
    \node[shape=circle,draw=black,fill=black,label=below:$e_3-e_4$] (E) at (0,5) {};
    \node[shape=circle,draw=black,fill=black,label=$e_2-e_3$] (F) at (2,5) {};
    \node[shape=circle,draw=black,fill=black,label=below:$\ell-128$] (G) at (4,5) {};
    \node[shape=circle,draw=black,fill=black,label=$e_8-e_9$] (H) at (6,5) {};
    \node[shape=circle,draw=black,fill=black,label=left:$e_1-e_2$] (I) at (2,4) {};
    \node[shape=circle,draw=black,fill=black,label=left:$\ell-167$] (J) at (2,3) {};

    \path [-](B) edge node[left] {} (C);
    \path [-](A) edge node[left] {} (C);
    \path [-](A) edge node[left] {} (B);
    
    \path [-](D) edge node[left] {} (E);
    \path [-](E) edge node[left] {} (F);
    \path [-](F) edge node[left] {} (G);
     \path [-](G) edge node[left] {} (H);
    \path [-](F) edge node[left] {} (I);
    \path [-](I) edge node[left] {} (J);

\end{tikzpicture}
\end{center}
will not yield an unexpected cubic.

In the $\wilde{A}_2^{\+4}$ diagram, we have an example of a pencil spanned by $2x^2z+yz^2$ and $xy^2+2yz^2$. This pencil has a base points of multiplicity 2 at $(0,0,1)$, $(0,1,0)$, and $(1,0,0)$ and has a base point of multiplicity 3 at $(1,1,1)$. We can create a two-dimensional linear system of cubics by introducing the third polynomial $(xz+2y^2)(x+2y)$. This linear system has two base points of multiplicity 2: one at $(0,0,1)$ and one at $(1,0,0)$. It also has a base point of multiplicity 3 at $(1,1,1)$. These seven points are all the base points of the linear system. Since unexpected cubics have a general double point, we must now determine whether this linear system has a general double point.

Let us examine the linear combination $$F(xz+2y^2)(x+2y)+G(2x^2z+yz^2)+H(xy^2+2yz^2)$$ and the generic point $P=(a,b,c)$. It is sufficient to reduce to the affine case $z=1$, and we want to determine whether there are values for $F$, $G$, and $H$ such that $$F(x+2y^2)(x+2y)+G(2x^2+y)+H(xy^2+2y)$$ has a double point at $(A,B)=(a/c,b/c)$. By shifting $x\mapsto x+A$ and $y\mapsto y+B$, we get that $(F,G,H)$ annihilates the constant and linear components of the resulting polynomial if and only if $$(F,G,H)\in\ker\begin{pmatrix}-A^2B+B^3&-A^2+B&AB^2-B\\-B^2-A-B&A&B^2\\AB-A&1&-AB-1\end{pmatrix}.$$ But this matrix has a determinant $$A^4B^2-A^3B^3-A^2B^4+A^4B+A^2B^3-AB^4-B^5+A^3B-AB^3+A^3\neq 0$$ and so for a general point $(A,B)$, the kernel of this matrix will be trivial, and so no linear combination of $(x+2y^2)(x+2y)$, $(2x^2+y)$, and $(xy^2+2y)$ will have a double point at $(A,B)$. Hence this two-dimensional linear system of cubics does not has a generic double point. Therefore these seven base points do not admit an unexpected cubic. We can repeat this process by instead creating a two-dimensional linear system that omits the double point at $(0,0,1)$ or at $(1,0,0)$, and we will similarly find that no set of seven base points will admit an unexpected cubic.

Next we will look at the $\wilde{A}_2\+\wilde{E}_6$ configuration. This is a pencil of cubics with base points $\alpha$ of multiplicity 2, $\beta$ of multiplicity 2, and $\gamma$ of multiplicity 5. We will denote the nine points as $\alpha^1$, $\alpha^2$, $\beta^1$, $\beta^2$, $\gamma^1$, $\gamma^2$, $\gamma^3$, $\gamma^4$, and $\gamma^5$, where $\alpha^1$, $\alpha^2$, and $\gamma^1$ are collinear, $\gamma^1$, $\gamma^2$, and $\beta^1$ are collinear, and $\beta^1$, $\beta^2$, and $\alpha^1$ are collinear.

The reason for these collinearity relations is this: if we observe the final relabeling of the Dynkin diagram $\wilde{A}_2\+\wilde{E}_6$, we see by identifying $\gamma^1,\dots,\gamma^5$ with $e_1,\dots,e_5$, $\alpha^1,\alpha^2$ with $e_6,e_7$, and $\beta^1,\beta^2$ with $e_8,e_9$, then the presence of the $(-2)$-curves $\ell-e_6-e_8-e_9$, $\ell-e_1-e_2-e_8$, and $\ell-e_1-e_6-e_7$ indicate the stated collinearity relations.

In order to get a two-dimensional linear system of cubics, we can choose to either take the seven base points $\alpha^1,\alpha^2,\gamma^1,\dots,\gamma^5$ or $\beta^1,\beta^2,\gamma^1,\dots,\gamma^5$.

In the first scenario, we want a cubic with a generic double point $P$ given seven base points $\alpha^1,\alpha^2,\gamma^1,\dots,\gamma^5$. In order for the resulting pencil to still be quasi-elliptic, $P$ needs to be on the line determined by $\gamma^1$ and $\gamma^2$, and so it turns out $P$ cannot be generic after all. If $P$ is not on this line, then we will not have the $\wilde{A}_2\+\wilde{E}_6$ diagram, and there is no quasi-elliptic Dynkin diagram for the configuration we would get.

Similarly, in the second scenario, we want a cubic with a generic double point $P$ given seven base points $\beta^1,\beta^2,\gamma^1,\dots,\gamma^5$. In order for the resulting pencil to still be quasi-elliptic, $P$ needs to be on the line determined by $\beta^1$ and $\beta^2$, and so it turns out $P$ cannot be generic after all. Therefore the quasi-elliptic fibration $\wilde{A}_2\+\wilde{E}_6$ does not yield an unexpected cubic.

Therefore these $\wilde{A}_2^{\+4}$ and $\wilde{A}_2\+\wilde{E}_6$ diagrams do not yield unexpected cubics. No other diagram of a blow-down of a quasi-elliptic fibration could yield an unexpected cubic because no other diagram involves a point being blown up exactly twice, which is necessary for the generic double point. Therefore there are no unexpected cubics in characteristic 3 coming from the quasi-elliptic fibrations enumerated above.

\end{proof}
\begin{cor}\label{no3}
Let $k$ be a field of characteristic 3. There are no unexpected cubic curves in $\P^2_k$.
\end{cor}
\begin{proof}
Suppose there exists an unexpected cubic in $\P^2_k$. Then there are seven (possibly infinitely-near) points $P_1,\dots,P_7\in\P^2_k$ such that there exists a cubic curve $C$ containing the seven points and is singular at a general point $P_8$. Then consider a cubic curve $C'$ containing $P_1,\dots,P_7,P_8$ and is singular at some point $Q'$, so then $C.C'=P_1+\cdots+P_7+2P_8$. Then we may form a pencil $\Pcal$ of cubic curves spanned by $C$ and $C'$ whose base locus is $P_1+\cdots+P_7+P_8^{(1)}+P_8^{(2)}$, where $P_8^{(2)}$ is an infinitely-near point to $P_8^{(1)}$ corresponding to the tangent line of $C'$ at $P_8\in\P^2_k$. 

The pencil $\Pcal$ thus admits a quasi-elliptic surface $X$ by blowing up the nine points of the base locus. Then $X$ must admit a Dynkin diagram of $\wilde{A}_2^{\+4}$, $\wilde{A}_2\+\wilde{E}_6$, or $\wilde{E}_8$. By Corollary \ref{extremal}, $X$ must be extremal and so has finitely many $(-1)$-curves. Furthermore, up to reflections across $(-2)$-classes (i.e. conjugation), each of the diagrams embed in the $(-2)$-curves of $X$ in a unique way by Theorem 5.5 of \cite{Dy}. Therefore $X$ must be isomorphic to one of the seven surfaces explored in this chapter. As we have already seen in Theorem \ref{thm3}, there are no unexpected cubics coming from the seven surfaces of this chapter, giving us a contradiction.
\end{proof}

\section{Appendix}
In Macaulay2 \cite{M2}, set $A$ equal to one of the below adjacency matrices, depending on the Dynkin diagram being analyzed.\\

\color{black}
\tiny

\texttt{M2A1D6=matrix\{\{-2,2,0,0,0,0,0,0,0,0,0,0,0,1,1\},\{2,-2,0,0,0,0,0,0,0,0,0,1,1,0,0\},\{0,0,-2,2,0,0,0,0,0,0,0,1,0,1,0\},\\\{0,0,2,-2,0,0,0,0,0,0,0,0,1,0,1\},\{0,0,0,0,-2,0,1,0,0,0,0,0,0,1,0\},\{0,0,0,0,0,-2,1,0,0,0,0,0,1,0,0\},\{0,0,0,0,1,1,-2,1,0,0,0,0,0,0,0\},\\\{0,0,0,0,0,0,1,-2,1,0,0,0,0,0,0\},\{0,0,0,0,0,0,0,1,-2,1,1,0,0,0,0\},\{0,0,0,0,0,0,0,0,1,-2,0,0,0,0,1\},\{0,0,0,0,0,0,0,0,1,0,-2,1,0,0,0\},\\
\{0,1,1,0,0,0,0,0,0,0,1,-1,0,0,0\},\{0,1,0,1,0,1,0,0,0,0,0,0,-1,0,0\},\{1,0,1,0,1,0,0,0,0,0,0,0,0,-1,0\},\{1,0,0,1,0,0,0,0,0,1,0,0,0,0,-1\}\}
}\\

\texttt{M2D4=matrix\{\{-2,0,0,0,1,0,0,0,0,0,1,0,0,0\},\{0,-2,0,0,0,1,0,0,0,0,0,1,0,0\},\{0,0,-2,0,0,1,0,0,0,0,0,0,1,0\},\\\{0,0,0,-2,0,1,0,0,0,0,0,0,0,1\},\{1,0,0,0,-2,0,0,1,1,1,0,0,0,0\},\{0,1,1,1,0,-2,1,0,0,0,0,0,0,0\},\{0,0,0,0,0,1,-2,0,0,0,1,0,0,0\},\\\{0,0,0,0,1,0,0,-2,0,0,0,1,0,0\},\{0,0,0,0,1,0,0,0,-2,0,0,0,1,0\},\{0,0,0,0,1,0,0,0,0,-2,0,0,0,1\},\{1,0,0,0,0,0,1,0,0,0,-1,0,0,0\},\\\{0,1,0,0,0,0,0,1,0,0,0,-1,0,0\},\{0,0,1,0,0,0,0,0,1,0,0,0,-1,0\},\{0,0,0,1,0,0,0,0,0,1,0,0,0,-1\}\}}\\

\texttt{MA1E7=matrix\{\{-2,2,0,0,0,0,0,0,0,0,0,1\},\{2,-2,0,0,0,0,0,0,0,0,1,0\},\{0,0,-2,1,0,0,0,0,0,0,1,0\},\{0,0,1,-2,1,0,0,0,0,0,0,0\},\\\{0,0,0,1,-2,1,0,0,0,0,0,0\},\{0,0,0,0,1,-2,1,1,0,0,0,0\},\{0,0,0,0,0,1,-2,0,0,0,0,0\},\{0,0,0,0,0,1,0,-2,1,0,0,0\},\{0,0,0,0,0,0,0,1,-2,1,0,0\},\\\{0,0,0,0,0,0,0,0,1,-2,0,1\},\{0,1,1,0,0,0,0,0,0,0,-1,0\},\{1,0,0,0,0,0,0,0,0,1,0,-1\}\}}\\

\texttt{MD8=matrix\{\{-2,0,1,0,0,0,0,0,0,1,0\},\{0,-2,1,0,0,0,0,0,0,0,0\},\{1,1,-2,1,0,0,0,0,0,0,0\},\{0,0,1,-2,1,0,0,0,0,0,0\},\\\{0,0,0,1,-2,1,0,0,0,0,0\},\{0,0,0,0,1,-2,1,0,0,0,0\},\{0,0,0,0,0,1,-2,1,1,0,0\},\{0,0,0,0,0,0,1,-2,0,0,1\},\{0,0,0,0,0,0,1,0,-2,0,0\},\\\{1,0,0,0,0,0,0,0,0,-1,0\},\{0,0,0,0,0,0,0,1,0,0,-1\}\}}\\

\texttt{ME8=matrix\{\{-2,1,0,0,0,0,0,0,0,0\},\{1,-2,1,0,0,0,0,0,0,0\},\{0,1,-2,1,1,0,0,0,0,0\},\{0,0,1,-2,0,0,0,0,0,0\},\{0,0,1,0,-2,1,0,0,0,0\},\\\{0,0,0,0,1,-2,1,0,0,0\},\{0,0,0,0,0,1,-2,1,0,0\},\{0,0,0,0,0,0,1,-2,1,0\},\{0,0,0,0,0,0,0,1,-2,1\},\{0,0,0,0,0,0,0,0,1,-1\}\}}\\


\texttt{MA2E6=matrix\{\{-2,1,0,1,0,1,0,0,0,0,0,0,0\},\{1,-2,1,0,0,0,0,0,0,0,0,0,0\},\{0,1,-2,0,0,0,0,0,0,0,1,0,0\},\{1,0,0,-2,1,0,0,0,0,0,0,0,0\},\\\{0,0,0,1,-2,0,0,0,0,0,0,1,0\},\{1,0,0,0,0,-2,1,0,0,0,0,0,0\},\{0,0,0,0,0,1,-2,0,0,0,0,0,1\},\{0,0,0,0,0,0,0,-2,1,1,1,0,0\},\\\{0,0,0,0,0,0,0,1,-2,1,0,1,0\},\{0,0,0,0,0,0,0,1,1,-2,0,0,1\},\{0,0,1,0,0,0,0,1,0,0,-1,0,0\},\{0,0,0,0,1,0,0,0,1,0,0,-1,0\},\\\{0,0,0,0,0,0,1,0,0,1,0,0,-1\}\}}\\

\normalsize
\color{black}
We will define a collection of objects and functions in Macaulay2 that one may use to verify our claims.\\

{\color{black}\texttt{L=toList(\{0..\#(entries A)-1\}\char`_0)};\\}

The above objects simply enumerates the vertices of $A$.\\

{\color{black}\texttt{blowdown=(M,i)->matrix(M\char`_i)*transpose(matrix(M\char`_i))+M;
}\\}

The above function takes an input matrix $M$ and a natural number $i$ and calculates the adjacency matrix associated with blowing down the vertex $i$, using the fact from Remark \ref{algo}.\\

{\color{black}\texttt{isValid=(M,i)->M\char`_i\char`_i==-1;
}\\}

The above function tests whether a given vertex represents a $(-1)$-curve, returning {\color{black}\texttt{true}} or {\color{black}\texttt{false}}.\\

{\color{black}\texttt{bdseq=S->(M=A; for i from 0 to \#S-1 do M=blowdown(M,S\char`_i); M)
}\\}

The above function receives a sequence of natural numbers $S$ and calculates the adjacency matrix induced by blowing down the original adjacency matrix $A$ at the sequence of vertices corresponding to $S$.\\

{\color{black}\texttt{(S=\{\{\}\}; for i from 0 to 8 do S=apply((unique apply((flatten\\ apply((apply(S,l->\{l,select(L,ll->isValid(bdseq(l),ll))\})),l->\\apply(l\char`_1,i->append(l\char`_0,i)))),l->sort(l))),i->(select((flatten\\ apply((apply(S,l->\{l,select(L,ll->isValid(bdseq(l),ll))\})),l->\\apply(l\char`_1,i->append(l\char`_0,i)))),l->sort(l)==i))\char`_0);)}\\}

The above recursive operation finds all sequences of blow-downs along nine vertices. The output is the list $S$, which comprises lists of nine natural numbers corresponding to a valid blow-down sequence. This operation can be broken down into the two simpler sub-operations.\\

{\color{black}\texttt{flatten apply((apply(*,l->{l,select(L,ll->isValid(bdseq(l),ll))})),\\l->apply(l\char`_1,i->append(l\char`_0,i)))}\\}

The above operation takes an input $*$ as a list of lists and-- for each constituent list-- identifies the indices of the $(-1)$-curves of the new surface attained by blowing down the curves in the input list. It then appends all those $(-1)$-curves to the existing list, creating a new list of lists of length one greater than the original. This is the primary engine behind the recursion, growing the list of all valid blow-down sequences as the recursion progresses. It can be made more efficient by the following operation.\\

{\color{black}\noindent \texttt{apply((unique apply(*,l->sort(l))),i->(select(*,l->sort(l)==i))\char`_0)}\\}

The above operation takes an input $*$ as a list of lists and extracts a single representative list for each set of numbers that appears, eliminating sequences that are only different by ordering. The purpose of this part is to limit the size of $S$ as the recursion progresses by reducing the number of redundant sequences.\\

Now we must reduce the list $S$ to unique representatives based on isomorphism classes. For each collection of vertices in $S$, we will analyze the $9\times 9$ adjacency matrix of the subgraph created by restricting our original graph to the collection of vertices we are blowing down. Isomorphic subgraphs will have permutation-similar adjacency matrices. We will therefore make use of the Permutations package for Macaulay2 and create a list of all permutations on nine elements.\\

{\color{black}\texttt{loadPackage "Permutations";\\ Perms=apply(permutations({1,2,3,4,5,6,7,8,9}),p->permutation(p));}\\}

Finally, the following command creates a list $T$ comprising a single representative for each isomorphism class, and outputs the size of $T$.\\

{\color{black}\texttt{(T=\{S\char`_0\}; for i from 1 to \#S-1 do if \\any(T,t->any(Perms,p->A\char`_(S\char`_i)\^{}(S\char`_i)==p*(A\char`_t\^{}t)*p))\\ then T=T else T=append(T,S\char`_i); \#T)}}


\printbibliography

\end{document}